\documentclass[a4paper,USenglish,thm-restate]{lipics-v2021}
\nolinenumbers
\hideLIPIcs

\usepackage{cmll,bussproofs,amsmath,mathdots,amsfonts,amstext,stmaryrd,mathrsfs,calc,xspace,mathdots,mathtools,booktabs}
\usepackage{thmtools}
\usepackage{amsthm}
\usepackage{amssymb}
\usepackage{tikz}

\newtheorem{Def}{Definition}
\newtheorem{Thm}[Def]{Theorem}
\newtheorem{Lem}[Def]{Lemma}
\newtheorem{Cor}[Def]{Corollary}
\newtheorem*{Problem}{Problem}
\newcommand{\jump}[1]{\ensuremath{[#1]}}
\renewcommand{\vec}[1]{\bar{\mathsf{#1}}}
\newcommand{\fz}{\xrightarrow{\stackrel{?}{=}\vec 0}}
\newcommand{\abvass}{ABVASS$_{\vec 0}$}
\newcommand{\ts}[1]{\ensuremath{\text{\textsc{#1}}}}
\newcommand{\tower}{\ts{Tower}}
\newcommand{\judge}{\mathrel{\triangleright}}
\newcommand{\tet}[2]{{#1}^{{\iddots\raisebox{1.2ex}{{\tiny{$#1$}}}}}\raisebox{.7ex}{$\}$\scriptsize {$#2\text{ times}$}}}

\usetikzlibrary{positioning}
\usetikzlibrary{arrows}
\usetikzlibrary{automata}
\tikzstyle{state}=[intermediate,minimum width=.40cm,inner sep=8.5pt]
\tikzstyle{rstate}=[state,rectangle,rounded corners=8pt,minimum height=.65cm]
\tikzstyle{every node}=[font=\small]
\tikzstyle{every edge}=[draw,>=stealth',shorten >=1pt,thick]
\tikzset{
	place/.style={
		circle,
		black,
		draw=black,
		fill=white,
		minimum size=5mm,
	},
	state/.style={
		circle,
		black,
		draw=black,
		fill=white,
		inner sep=0.9pt,
		minimum size=5mm,
	},
	intermediate/.style={
		circle,
		black,
		draw=black,
		fill=white,
		minimum size=3mm,
	},
}

\title{Tower-Complete Problems in Contraction-Free Substructural Logics} 
\author{Hiromi Tanaka}{Keio University, Tokyo, Japan}{hiromitanaka@keio.jp}{https://orcid.org/0000-0002-2941-0850}{}

\authorrunning{H. Tanaka}

\Copyright{Hiromi Tanaka}


\funding{This work was partially supported by JST, CREST Grant Number JPMJCR2114, Japan.}

\acknowledgements{I would like to thank Koji Mineshima for helpful comments on the introduction section. I am also grateful to the anonymous reviewers for useful suggestions, which improve some proofs and deepen the discussion in Section~\ref{conclusion}.}

\begin{document}
\maketitle
\begin{abstract}
We investigate the non-elementary computational complexity of a family of substructural logics without contraction. 
With the aid of the technique pioneered by Lazi\'c and Schmitz (2015), we show that the deducibility problem for full Lambek calculus with exchange and weakening ($\mathbf{FL}_{\mathbf{ew}}$) is not in \ts{Elementary} (i.e., the class of decision problems that can be decided in time bounded by an elementary recursive function), but is in PR (i.e., the class of decision problems that can be decided in time bounded by a primitive recursive function). 
More precisely, we show that this problem is complete for \tower, which is a non-elementary complexity class forming a part of the fast-growing complexity hierarchy introduced by Schmitz (2016). The same complexity result holds even for deducibility in BCK-logic, i.e., the implicational fragment of $\mathbf{FL}_{\mathbf{ew}}$. We furthermore show the \tower-completeness of the provability problem for elementary affine logic, which was proved to be decidable by Dal Lago and Martini (2004).
\end{abstract}

\section{Introduction}
\label{intro}
The term ``substructural logic''~\cite{GJKO07,Ono03,Res00} is an umbrella term for a family of logics that limit the use of some of the structural rules.
Substructural logics encompass a wide range of non-classical logics (e.g., intuitionistic, classical, relevance, paraconsistent and multi-valued logics), and thus are discussed in several distinct areas close to mathematical logic such as philosophy, linguistics and computer science.
In any of those research fields, one of the major topics is to settle the computational complexity of the \emph{provability problem} for a logic, i.e., the problem of whether a given formula is provable in the logic.
There are many seminal papers concerning this subject; see, e.g., \cite{BLR21,HT11,LS15,LMSS92,Pen06,Sch216,Sta79,Urq99}.

It is no surprise that a more general problem can be considered for a given logic. 
The \emph{deducibility problem} for the logic asks for a given finite set $\Phi$ of formulas and a given formula $A$ whether $A$ is provable in the logic augmented with $\Phi$ as a set of non-logical axioms.
In the setting of classical and intuitionistic logic, the notion of deducibility is reduced to provability via the deduction theorem. 
As a result, the deducibility problem for intuitionistic (resp.\ classical) propositional logic is complete for \textsc{Pspace} (resp.\ coNP) as with the provability problem.
On the other hand, since most of substructural logics do not admit the deduction theorem, there is no guarantee that these two problems are inter-reducible to each other.
For this reason, it is important to distinguish them in the framework of substructural logic.
In fact, some substructural logics have a critical ``gap'' between the complexity of provability and the complexity of deducibility.
The so-called \emph{Lambek calculus} is a striking example of such logics. 
Buszkowski~\cite{Bus82} showed that its deducibility is undecidable, but later on Pentus~\cite{Pen06} proved that its provability is NP-complete. 

\paragraph*{Main contribution}
This paper aims at clarifying the non-elementary computational complexity of deducibility in contraction-free substructural logics.
So far, such a topic has not been sufficiently explored while some earlier papers investigated a non-primitive recursive complexity of weakening-free substructural logics (i.e., relevance logics); see~\cite{LS15,Sch216,Urq99}.

\emph{Full Lambek calculus with exchange and weakening} ($\mathbf{FL}_{\mathbf{ew}}$), i.e., intuitionistic logic without the rule of contraction, is one of the most basic contraction-free logics. 
The deducibility problem for this logic is known to be decidable, thanks to the finite embeddability property of FL$_{ew}$-algebras shown by Blok and van~Alten~\cite{BvA02,BvA05}. 
However, its exact complexity remained open. 
Hence the following natural questions arise:
\begin{quote}
\begin{itemize}
    \item \emph{Is there a primitive recursive algorithm---i.e., one whose runtime is bounded by a primitive recursive function---for the deducibility problem for $\mathbf{FL}_{\mathbf{ew}}$?}
    \item \emph{If so, is there an elementary recursive algorithm—i.e., one whose runtime is bounded by a tower of exponentials of fixed height---for the problem?}
\end{itemize}
\end{quote}
We answer ``yes'' to the first question, but provide a negative answer to the second question.
To be precise, we show that the problem is actually complete for the class \textsc{Tower} (Corollary~\ref{main}). 
This class forms a part of the \emph{fast-growing complexity hierarchy} introduced by Schmitz~\cite{Sch16}, and roughly speaking, is located between \textsc{Elementary} (i.e., the class of problems decidable in elementary time) and \textsc{PR} (i.e., the class of problems that can be solved in time bounded by a primitive recursive function). 
As a consequence, it turns out that the aforementioned ``gap''  also lies between provability and deducibility in $\mathbf{FL}_{\mathbf{ew}}$; the provability problem for $\mathbf{FL}_{\mathbf{ew}}$ is \ts{Pspace}-complete, cf.\ \cite{HT11}. 
 
We stress that the same holds even when almost all the logical connectives are removed from $\mathbf{FL}_{\mathbf{ew}}$.
We also prove that the deducibility problem for \emph{BCK-logic} \cite{IT78,OK85}, i.e., the implicational fragment of $\mathbf{FL}_{\mathbf{ew}}$, is \tower-complete (Corollary~\ref{main}). 
This is in sharp contrast to the NP-completeness of provability in BCK-logic (Corollary~\ref{np}).

\paragraph*{Proof overview}
To show the \textsc{Tower}-membership of deducibility in $\mathbf{FL}_{\mathbf{ew}}$, we prove that there are reductions: 
\begin{bracketenumerate}
    \item from deducibility in $\mathbf{FL}_{\mathbf{ew}}$ to provability in a variant of intuitionistic affine logic (denoted by $\mathbf{ILZW}'$),
    \item from the provability problem for $\mathbf{ILZW}'$ to the lossy reachability problem for \emph{alternating branching vector addition systems with states} (ABVASS, for short). 
\end{bracketenumerate}
The first reduction is quite similar to the one used in the famed proof of the undecidability of propositional linear logic by Lincoln, Mitchell, Scedrov and Shankar~\cite{LMSS92}. 
The second reduction is substantially inspired by Lazi\'c and Schmitz~\cite{LS15}. 
Due to the \tower-completeness of lossy reachability in ABVASS, shown in~\cite{LS15}, we obtain the membership in \tower\ of deducibility in $\mathbf{FL}_{\mathbf{ew}}$.

In order to show the \textsc{Tower}-hardness, we introduce the notion of \emph{$\oc$-prenex implicational sequent}. 
It is a slight modification of $\oc$-prenex sequents which Terui introduced in his PhD thesis~\cite{Ter02}. 
We prove the \tower-hardness of a restricted version of the provability problem for intuitionistic affine logic, which asks whether a given $\oc$-prenex implicational sequent is provable in intuitionistic affine logic. 
We obtain the desired result by showing that this problem can be reduced into deducibility in $\mathbf{FL}_{\mathbf{ew}}$. 

\paragraph*{Provability (type inhabitation) in elementary affine logic}
As a by-product resulting from our methods, we provide the precise complexity of provability (not of deducibility) in \emph{propositional elementary affine logic}~\cite{Asp98}. 
Its name comes from the fact that it characterizes elementary recursive computation in the paradigm of proofs-as-programs; see also~\cite{DJ03,Gir98}. 
Although this logic is seemingly just an extension of BCK-logic (and $\mathbf{FL}_{\mathbf{ew}}$) by a sort of modal storage operator, it is exploited for a variety of purposes, e.g., to characterize the class P and the exponential time hierarchy~\cite{Bai15}, to formulate a consistent naive set theory with a rich computational power~\cite{Ter04}. 

In most situations, elementary affine logic is treated as a type system rather than a purely logical system. 
Accordingly, as with many other type systems, some decision problems can be considered, i.e., typability, type checking and type inhabitation (provability). 
For instance, Coppola and Martini~\cite{CM06} showed the decidability of typability in the $\{\multimap,\oc\}$-fragment of intuitionistic elementary affine logic. 

On the other hand, of particular interest to us is the provability problem for elementary affine logic.
Dal Lago and Martini~\cite{DM04} showed that provability in a classical variant of elementary affine logic is decidable.
However, there are no known upper and lower bounds for that problem.
We refine and extend the existing decidability result by showing the \textsc{Tower}-completeness of some variants of elementary affine logic (Section~\ref{towerellw}). 
It should be noted that such a non-elementary aspect of elementary affine logic does not conflict with its elementary character which comes from the proofs-as-programs correspondence.

\paragraph*{Organization of this paper}
In the next section, we review various contraction-free logics in a step-by-step manner, and define a useful translation from classical affine logic into intuitionistic affine logic.
A large portion of Sections~\ref{abvass} and~\ref{tower} is taken from~\cite[Section~3]{LS15}.
In Section~\ref{abvass}, we summarize some basic notions involved in ABVASSs. 
Section~\ref{tower} is devoted to a short discussion about the existing complexity results which are crucial in proving the main claims in Sections~\ref{encodingoflogic} and \ref{encodingofbvass}.
We prove the main results in Sections~\ref{encodingoflogic} and \ref{encodingofbvass}. 
In Section~\ref{conclusion}, we conclude the paper with some remarks on the complexity status of other substructural logics.
The appendices include some proofs omitted in the main text.

\section{Contraction-free substructural logics}\label{sequent}
\subsection{Sequent calculi for contraction-free substructural logics}\label{sequentcalculus}

For convenience, we start with the formal definition of a sequent calculus for \emph{intuitionistic affine logic with bottom}, denoted by $\mathbf{ILZW}$. 
It is merely the extension of Troelstra's $\mathbf{ILZ}$ by the rule of left-weakening, cf.\ \cite{LS15,Tro92}.
The \emph{language} $\mathcal{L}$ of $\mathbf{ILZW}$ contains logical connectives $\with,\oplus,\otimes,\multimap$ of arity 2, $\oc$ of arity 1, and $\mathbf{1},\top,\bot,\mathbf{0}$ of arity 0. 
We fix a countable set of propositional variables $V=\{p,q,r,\ldots\}$. 
An \emph{intuitionistic $\mathcal{L}$-formula} is built from propositional variables using connectives in $\mathcal{L}$. 
For brevity, parentheses in formulas are omitted when confusion is unlikely. 
Throughout this paper, metavariables $A,B,C,\ldots$ range over formulas and $\Gamma,\Delta,\Sigma,\ldots$ over finite multisets of formulas. 
By abuse of notation, we simply write $A$ for the singleton of a formula $A$.
The multiset sum of $\Gamma$ and $\Delta$ is denoted by $\Gamma,\Delta$. We write $\oc \Gamma$ for the multiset obtained by prefixing each formula in $\Gamma$ with exactly one $\oc$. 
An \emph{intuitionistic $\mathcal{L}$-sequent} is an expression of the form $\Gamma \vdash \Pi$, where $\Gamma$ is a finite multiset of intuitionistic $\mathcal{L}$-formulas, and
$\Pi$ is a stoup, i.e., either an intuitionistic $\mathcal{L}$-formula or the empty multiset $\varepsilon$. 
We always denote an intuitionistic $\mathcal{L}$-sequent of the form $\varepsilon \vdash \Pi$ (resp.\ $\Gamma \vdash \varepsilon$) by $\vdash\Pi$ (resp.\ $\Gamma\vdash$). 
The sequent calculus for $\mathbf{ILZW}$ consists of the inference rules depicted in Figure~\ref{inf1}. 
\begin{figure}
\begin{center}
\begin{tabular}{ccc}
\AxiomC{}
\RightLabel{(Init)}
\UnaryInfC{$A \vdash A$}
\DisplayProof
&
\AxiomC{}
\RightLabel{($\mathbf{1}$R)}
\UnaryInfC{$\vdash \mathbf{1}$}
\DisplayProof
&
\AxiomC{}
\RightLabel{($\bot$L)}
\UnaryInfC{$\bot \vdash$}
\DisplayProof
\\[1.5em]
\AxiomC{}
\RightLabel{($\top$R)}
\UnaryInfC{$\Gamma \vdash \top$}
\DisplayProof
&
\AxiomC{}
\RightLabel{($\mathbf{0}$L)}
\UnaryInfC{$\mathbf{0},\Gamma \vdash \Pi$}
\DisplayProof
&
\AxiomC{$\Gamma \vdash A$}
\AxiomC{$A,\Delta \vdash \Pi$}
\RightLabel{(Cut)}
\BinaryInfC{$\Gamma,\Delta \vdash \Pi$}
\DisplayProof
\\[1.5em]
\AxiomC{$\Gamma \vdash \Pi$}
\RightLabel{($\mathbf{1}$L)}
\UnaryInfC{$\mathbf{1},\Gamma \vdash\Pi$}
\DisplayProof
&
\AxiomC{$\Gamma \vdash$}
\RightLabel{($\bot$R)}
\UnaryInfC{$\Gamma \vdash \bot$}
\DisplayProof
&
\AxiomC{$\Gamma \vdash A$}
\AxiomC{$B,\Delta \vdash \Pi$}
\RightLabel{($\multimap$L)}
\BinaryInfC{$A\multimap B,\Gamma,\Delta \vdash \Pi$}
\DisplayProof
\\[1.5em]
\AxiomC{$A,\Gamma \vdash B$}
\RightLabel{($\multimap$R)}
\UnaryInfC{$\Gamma \vdash A\multimap B$}
\DisplayProof
&
\AxiomC{$A,B,\Gamma \vdash \Pi$}
\RightLabel{($\otimes$L)}
\UnaryInfC{$A\otimes B,\Gamma \vdash \Pi$}
\DisplayProof
&
\AxiomC{$\Gamma\vdash A$}
\AxiomC{$\Delta \vdash B$}
\RightLabel{($\otimes$R)}
\BinaryInfC{$\Gamma,\Delta \vdash A \otimes B$}
\DisplayProof
\\[1.5em]
\AxiomC{$A,\Gamma\vdash\Pi$}
\RightLabel{($\with$L1)}
\UnaryInfC{$A\with B,\Gamma\vdash\Pi$}
\DisplayProof
&
\AxiomC{$B,\Gamma\vdash\Pi$}
\RightLabel{($\with$L2)}
\UnaryInfC{$A\with B,\Gamma\vdash\Pi$}
\DisplayProof
&
\AxiomC{$\Gamma\vdash A$}
\AxiomC{$\Gamma\vdash B$}
\RightLabel{($\with$R)}
\BinaryInfC{$\Gamma \vdash A \with B$}
\DisplayProof
\\[1.5em]
\AxiomC{$\Gamma\vdash A$}
\RightLabel{($\oplus$R1)}
\UnaryInfC{$\Gamma\vdash A\oplus B$}
\DisplayProof
&
\AxiomC{$\Gamma\vdash B$}
\RightLabel{($\oplus$R2)}
\UnaryInfC{$\Gamma\vdash A\oplus B$}
\DisplayProof
&
\AxiomC{$A,\Gamma\vdash \Pi$}
\AxiomC{$B,\Gamma\vdash \Pi$}
\RightLabel{($\oplus$L)}
\BinaryInfC{$A\oplus B,\Gamma \vdash \Pi$}
\DisplayProof
\\[1.5em]
\multicolumn{3}{c}{
\begin{tabular}{c@{\hspace{2em}}c@{\hspace{2em}}c@{\hspace{2em}}c}
\AxiomC{$A,\Gamma\vdash\Pi$}
\RightLabel{($\oc$D)}
\UnaryInfC{$\oc A,\Gamma\vdash\Pi$}
\DisplayProof
&
\AxiomC{$\oc\Gamma\vdash A$}
\RightLabel{($\oc$P)}
\UnaryInfC{$\oc\Gamma\vdash \oc A$}
\DisplayProof
&
\AxiomC{$\Gamma\vdash\Pi$}
\RightLabel{(W)}
\UnaryInfC{$A,\Gamma\vdash\Pi$}
\DisplayProof
&
\AxiomC{$\oc A,\oc A,\Gamma\vdash\Pi$}
\RightLabel{($\oc$C)}
\UnaryInfC{$\oc A,\Gamma\vdash\Pi$}
\DisplayProof
\end{tabular}
}
\end{tabular}
\end{center}
\caption{Inference rules of $\mathbf{ILZW}$; $A,B$ range over intuitionistic $\mathcal{L}$-formulas and $\Gamma,\Delta$ range over finite multisets of intuitionistic $\mathcal{L}$-formulas, and $\Pi$ ranges over stoups.}
\label{inf1}
\end{figure}
A \emph{proof} of a sequent $\Gamma \vdash \Pi$ in $\mathbf{ILZW}$ is defined in the usual manner. 
We furthermore define another variant of intuitionistic affine logic by adding the following \emph{right-weakening} rule (W') to $\mathbf{ILZW}$:
\begin{center}
    \AxiomC{$\Gamma \vdash$}
    \RightLabel{(W')}
    \UnaryInfC{$\Gamma \vdash A$}
    \DisplayProof
\end{center}
The resulting system is denoted by $\mathbf{ILZW}'$.

Let $\mathcal{K}$ be a non-empty subset of $\mathcal{L}$. 
An \emph{intuitionistic $\mathcal{K}$-formula} is an intuitionistic formula containing only logical connectives from $\mathcal{K}$. 
An \emph{intuitionistic $\mathcal{K}$-sequent} is an intuitionistic sequent consisting only of intuitionistic $\mathcal{K}$-formulas. 
The \emph{$\mathcal{K}$-fragment} of $\mathbf{ILZW}$ (resp.\ $\mathbf{ILZW'}$) is the sequent calculus obtained from $\mathbf{ILZW}$ (resp.\ $\mathbf{ILZW}'$) by dropping all
the inference rules concerning connectives not in $\mathcal{K}$. 

Each of the logical systems within the scope of this paper is a fragment of $\mathbf{ILZW}$ or $\mathbf{ILZW}'$. 
We list them below:
\begin{itemize}
	\item  \emph{Full Lambek calculus with exchange and left-weakening} ($\mathbf{FL}_{\mathbf{ei}}$) is the $\{\otimes,\multimap,\with,\oplus, \mathbf{1},\bot\}$-fragment of $\mathbf{ILZW}$.
	\item \emph{Full Lambek calculus with exchange and weakening} ($\mathbf{FL}_{\mathbf{ew}}$) is the $\{\otimes,\multimap,\with,\oplus, \mathbf{1},\bot\}$-fragment of $\mathbf{ILZW}'$.
	\item The \emph{positive fragment} ($\mathbf{FL}^+_{\mathbf{ei}}$) of $\mathbf{FL}_{\mathbf{ei}}$ is the $\{\otimes,\multimap,\with,\oplus, \mathbf{1}\}$-fragment of $\mathbf{ILZW}$.
	\item \emph{BCK logic} ($\mathbf{BCK}$) is the implicational fragment (i.e., $\{\multimap\}$-fragment) of $\mathbf{ILZW}$.
\end{itemize}
Unfortunately, our notation is considerably different from the notation widely employed in the substructural logic community; we refer the reader to \cite[Table 2.1]{GJKO07} for the notational correspondence between linear logic and substructural logic.

We next formulate right-hand sided sequent calculi for \emph{classical affine logic} ($\mathbf{LLW}$), cf.\ \cite{Gir95,Laf97}. 
For our purpose here, we again employ the countable set $V$ of propositional variables, and introduce their duals $V^{\bot}=\{p^{\bot}, q^{\bot}, r^{\bot},\ldots\}$. 
Elements in $V \cup V^{\bot}$ are often referred to as \emph{literals}. 
The language $\mathcal{L}_C$ consists of binary operation symbols $\with,\oplus,\otimes,\parr$, and constants $\mathbf{1},\top,\bot,\mathbf{0}$, and unary operation symbols $\oc,\wn$. 
Given a sublanguage $\mathcal{K}$ of $\mathcal{L}_C$, \emph{Classical $\mathcal{K}$-formulas} are built up from literals using operation symbols in $\mathcal{K}$.
For each classical $\mathcal{L}_C$-formula $A$, we inductively define the formula $A^{\bot}$ by the de Morgan duality; $(p^{\bot})^{\bot}=p, (A \with B)^{\bot}=A^{\bot} \oplus B^{\bot}, (A \oplus B)^{\bot}=A^{\bot} \with B^{\bot}, (A \otimes B)^{\bot}=A^{\bot} \parr B^{\bot},(A \parr B)^{\bot}=A^{\bot} \otimes B^{\bot},(\oc A)^{\bot}=\wn A^{\bot},(\wn A)^{\bot}=\oc A^{\bot},\mathbf{1}^{\bot}=\bot,\bot^{\bot}=\mathbf{1},\top^{\bot}=\mathbf{0},\mathbf{0}^{\bot}=\top$. 
It is easy to see that $A=A^{\bot\bot}$ for any classical $\mathcal{L}_C$-formula $A$. 
A \emph{classical $\mathcal{K}$-sequent} is an expression of the form $\vdash \Gamma$, where $\Gamma$ is a finite multiset of classical $\mathcal{K}$-formulas. 
The inference rules of $\mathbf{LLW}$ are presented in Figure~\ref{inf2}. 
\begin{figure}
\begin{center}
\begin{tabular}{ccc}
\multicolumn{3}{c}{
\begin{tabular}{c@{\hspace{2em}}c@{\hspace{2em}}c@{\hspace{2em}}c}
\AxiomC{}
\RightLabel{(Init)}
\UnaryInfC{$\vdash A,A^{\bot}$}
\DisplayProof
&
\AxiomC{}
\RightLabel{($\mathbf{1}$)}
\UnaryInfC{$\vdash \mathbf{1}$}
\DisplayProof
&
\AxiomC{}
\RightLabel{($\top$)}
\UnaryInfC{$\vdash\Gamma,\top$}
\DisplayProof
&
\AxiomC{$\vdash\Gamma,A$}
\AxiomC{$\vdash A^{\bot},\Delta$}
\RightLabel{(Cut)}
\BinaryInfC{$\vdash\Gamma,\Delta$}
\DisplayProof
\end{tabular}
}
\\[1.5em]
\AxiomC{$\vdash\Gamma$}
\RightLabel{($\bot$)}
\UnaryInfC{$\vdash\Gamma,\bot$}
\DisplayProof
&
\AxiomC{$\vdash\Gamma,A$}
\AxiomC{$\vdash\Delta,B$}
\RightLabel{($\otimes$)}
\BinaryInfC{$\vdash\Gamma,\Delta,A\otimes B$}
\DisplayProof
&
\AxiomC{$\vdash\Gamma,A,B$}
\RightLabel{($\parr$)}
\UnaryInfC{$\vdash\Gamma,A\parr B$}
\DisplayProof
\\[1.5em]
\AxiomC{$\vdash\Gamma,A$}
\AxiomC{$\vdash\Gamma,B$}
\RightLabel{($\with$)}
\BinaryInfC{$\vdash\Gamma,A\with B$}
\DisplayProof
&
\AxiomC{$\vdash\Gamma,A$}
\RightLabel{($\oplus$1)}
\UnaryInfC{$\vdash\Gamma,A\oplus B$}
\DisplayProof
&
\AxiomC{$\vdash\Gamma,B$}
\RightLabel{($\oplus$2)}
\UnaryInfC{$\vdash\Gamma,A\oplus B$}
\DisplayProof
\\[1.5em]
\multicolumn{3}{c}{
\begin{tabular}{c@{\hspace{2em}}c@{\hspace{2em}}c@{\hspace{2em}}c}
\AxiomC{$\vdash\Gamma,A$}
\RightLabel{($\wn$)}
\UnaryInfC{$\vdash\Gamma,\wn A$}
\DisplayProof
&
\AxiomC{$\vdash\wn\Gamma,A$}
\RightLabel{($\oc$)}
\UnaryInfC{$\vdash\wn\Gamma,\oc A$}
\DisplayProof
&
\AxiomC{$\vdash\Gamma$}
\RightLabel{(W)}
\UnaryInfC{$\vdash\Gamma,A$}
\DisplayProof
&
\AxiomC{$\vdash\Gamma,\wn A,\wn A$}
\RightLabel{($\wn$C)}
\UnaryInfC{$\vdash\Gamma,\wn A$}
\DisplayProof
\end{tabular}
}
\end{tabular}
\end{center}
\caption{Inference rules of $\mathbf{LLW}$; metavariables $A,B$ range over classical $\mathcal{L}_C$-formulas, and $\Gamma,\Delta$ over finite multisets of $\mathcal{L}_C$-formulas.}
\label{inf2}
\end{figure}

As with the intuitionistic sequent systems discussed earlier, various fragments of $\mathbf{LLW}$ can be defined in the usual manner. 
Among such fragments, of importance to us is the $\{\otimes,\parr,\with,\oplus,\mathbf{1},\bot\}$-fragment of $\mathbf{LLW}$, called \emph{involutive full Lambek calculus with
exchange and weakening} ($\mathbf{InFL}_{\mathbf{ew}}$). 

In \cite{Gir98} Girard proposed a logic which captures elementary recursive computation, called \emph{elementary linear logic} ($\mathbf{ELL}$).    
We review here some affine variants of $\mathbf{ELL}$, following \cite{Bai15,CM06,DM04}. 
\emph{Intuitionistic elementary affine logic with bottom} ($\mathbf{IEZW}$) is obtained from $\mathbf{ILZW}$ by dropping the rules of ($\oc$D) and ($\oc$P) and by adding the following functorial promotion rule:
\begin{prooftree}
    \AxiomC{$\Gamma \vdash A$}
    \RightLabel{($\oc$F)}
    \UnaryInfC{$\oc\Gamma \vdash \oc A$}
\end{prooftree}
Similarly, \emph{classical elementary affine logic} $(\mathbf{ELLW})$ is obtained from $\mathbf{LLW}$ by deleting the rules of $(\wn)$ and $(\oc)$ and by adding the following rule: 
\begin{prooftree}
    \AxiomC{$\vdash\Gamma,A$}
    \RightLabel{(F)}
    \UnaryInfC{$\vdash\wn\Gamma,\oc A$}
\end{prooftree}
It is easy to see that $\mathbf{IEZW}$ (resp.\ $\mathbf{ELLW}$) is a subsystem of $\mathbf{ILZW}$ (resp.\ $\mathbf{LLW}$), i.e., every sequent that is provable in $\mathbf{IEZW}$ (resp.\ $\mathbf{ELLW}$) is provable in $\mathbf{ILZW}$ (resp.\ $\mathbf{LLW}$).  
Henceforce, we write $\mathcal{L}^+$ for the language $\mathcal{L} \setminus \{\bot\}$. 
The notation $\mathbf{ILLW}$ (resp.\ $\mathbf{IELW}$) is used to denote the $\mathcal{L}^+$-fragment of $\mathbf{ILZW}$ (resp.\ $\mathbf{IEZW}$).
Every sequent in such $\bot$-free logical systems is of the form $\Gamma \vdash A$. 

Let $\mathbf{L}$ be one of the sequent calculi described so far and $\Phi$ a set of formulas in $\mathbf{L}$. 
We write $\mathbf{L}[\Phi]$ for the sequent calculus obtained from $\mathbf{L}$ by adding $\vdash B$ as an initial sequent for every $B\in \Phi$.  
In this paper, we consider the following two types of decision problems for $\mathbf{L}$.

\begin{Problem}[Provability in $\mathbf{L}$]\hfill
\begin{description}
\item[Instance:] A formula $F$ in $\mathbf{L}$.
\item[Question:] Is the sequent $\vdash F$ provable in $\mathbf{L}$?
\end{description}
\end{Problem}

\begin{Problem}[Deducibility in $\mathbf{L}$]\hfill
\begin{description}
\item[Instance:] A finite set $\Phi \cup \{F\}$ of formulas in $\mathbf{L}$.
\item[Question:] Is the sequent $\vdash F$ provable in $\mathbf{L}[\Phi]$?
\end{description}
\end{Problem}

Our argument depends heavily on the following cut-elimination theorem, as we will see in the remaining sections:
\begin{Thm}[cf.\ \cite{DM04,Laf97,OT99}]\label{cut}
	The sequent calculi for $\mathbf{ILLW}$, $\mathbf{ILZW}$, $\mathbf{ILZW}'$, $\mathbf{LLW}$, $\mathbf{IELW}$, $\mathbf{IEZW}$, and $\mathbf{ELLW}$ all enjoy cut-elimination.
\end{Thm}

As far as we know, for instance, the cut-elimination for $\mathbf{ILZW}'$ has not been settled. 
The reader can however show this without great difficulty, using a proof-theoretic or algebraic manner; see also \cite{LMSS92,Oka02,Tro92} for technical details on cut-elimination in linear logic. 
We thus omit the proof in this paper. 

Of course, the cut-elimination theorem also holds for various fragments of the systems stated in Theorem~\ref{cut}, e.g., $\mathbf{BCK}$, $\mathbf{FL}_{\mathbf{ei}}$, the $\{\multimap,\oc\}$-fragment of $\mathbf{ILZW}$.

\subsection{Translation from classical affine logic to intuitionistic affine logic}\label{simple}
\label{translation}
In this subsection, we present an efficient (i.e., polynomial-time) translation from $\mathbf{LLW}$ (resp.\ $\mathbf{ELLW}$) into $\mathbf{ILLW}$ (resp.\ $\mathbf{IELW}$).
It is a modification of Laurent's \emph{parametric negative translation} from classical linear logic to intuitionistic linear logic; see \cite[Definition~2.2]{Lau18}. 

Let us fix an intuitionistic $\mathcal{L}$-formula $F$. 
Given a classical $\mathcal{L}_C$-formula $A$, we inductively define the intuitionistic $\mathcal{L}$-formula $A^{[F]}$ as follows:
\begin{align*}
p^{\jump{F}}&:=\neg_F p & (p^{\bot})^{\jump{F}}&:=p\\
\mathbf{1}^{\jump{F}}&:=\neg_F \mathbf{1} & \bot^{\jump{F}}&:=\mathbf{1}\\
\top^{\jump{F}}&:=\mathbf{0} & \mathbf{0}^{\jump{F}}&:=\neg_F \mathbf{0}\\
(B \otimes C)^{\jump{F}}&:=\neg_F B^{\jump{F}} \multimap C^{\jump{F}} & (B \parr C)^{\jump{F}}&:=\neg_F(B^{\jump{F}}\multimap \neg_F C^{\jump{F}}) \\
(B \with C)^{\jump{F}}&:=B^{\jump{F}} \oplus C^{\jump{F}} & (B \oplus C)^{\jump{F}}&:=\neg_F(\neg_F B^{\jump{F}}\oplus \neg_F C^{\jump{F}}) \\
(\oc B)^{\jump{F}}&:=\neg_F \oc \neg_F B^{\jump{F}} & (\wn B)^{\jump{F}}&:=\oc B^{\jump{F}}
\end{align*}
where $\neg_F A$ is an abbreviation for $A \multimap F$. 
We can show the following theorem; see Appendix~\ref{detailstranslation} for a proof.

\begin{restatable}{Thm}{translation}\label{illtoll5}
	Let $\vdash \Gamma$ be a classical $\mathcal{L}_C$-sequent and $x$ a fresh propositional variable not occurring in $\Gamma$.
	\begin{bracketenumerate}
		\item $\vdash \Gamma$ is provable in $\mathbf{LLW}$ if and only if $\Gamma^{[x]} \vdash x$ is provable in $\mathbf{ILLW}$.
		\item $\vdash \Gamma$ is provable in $\mathbf{ELLW}$ if and only if $\Gamma^{[x]} \vdash x$ is provable in $\mathbf{IELW}$.
	\end{bracketenumerate}
\end{restatable}

This translation is convenient to show the complexity of the contraction-free logics that we deal with, e.g., the NP-completeness of the provability problem for $\mathbf{BCK}$.

\begin{Cor}\label{np}
	The provability problem for $\mathbf{BCK}$ is \textsc{NP}-complete.
\end{Cor}
\begin{proof}
Membership in NP is an immediate consequence of cut elimination for $\mathbf{BCK}$.
The proof is based on that of~\cite[Lemma~5.3]{LMSS92}.
In any cut-free proof in the system, the only applicable rules are (Init), ($\multimap$L), ($\multimap$R) and (W). 
Thus each subformula occurring in the endsequent is analyzed at most once in such a proof-tree. 
This means that, the size of a cut-free proof in the system is polynomially bounded in the size of the endsequent.
Hence the problem is in NP.\footnote{This was already pointed out in~\cite[p.\ 71]{Han17}.} 

For the hardness, we construct a polynomial-time reduction from provability in the $\{\otimes,\parr\}$-fragment of $\mathbf{LLW}$ (i.e., the constant-free fragment of multiplicative classical affine logic)  into provability in $\mathbf{BCK}$. 
The NP-completeness of the former is shown by Lincoln-Mitchell-Scedrov-Shankar~\cite{LMSS92}, and Kanovich~\cite{Kan94}.
Let $A$ be a classical $\{\otimes,\parr\}$-formula and $x$ a fresh variable not in $A$. 
Our goal is to show that $\vdash A$ is provable in the $\{\otimes,\parr\}$-fragment of $\mathbf{LLW}$ if and only if $\vdash \neg_x A^{[x]}$ is provable in $\mathbf{BCK}$.
As a consequence of cut elimination for $\mathbf{LLW}$, we can easily show that $\mathbf{LLW}$ is conservative over its $\{\otimes,\parr\}$-fragment.
That is, $\vdash A$ is provable in the $\{\otimes,\parr\}$-fragment of $\mathbf{LLW}$ if and only if $\vdash A$ is provable in $\mathbf{LLW}$.
By Theorem~\ref{illtoll5}, $\vdash A$ is provable in $\mathbf{LLW}$ if and only if $A^{[x]}\vdash x$ is provable in $\mathbf{ILLW}$. 
Here $A^{[x]}\vdash x$ is an intuitionistic $\{\multimap\}$-sequent.
Again, by the cut elimination theorem for $\mathbf{ILLW}$, we know that $\mathbf{ILLW}$ is conservative over $\mathbf{BCK}$; hence $A^{[x]}\vdash x$ is provable in $\mathbf{ILLW}$ if and only if $A^{[x]}\vdash x$ is provable in $\mathbf{BCK}$.
By the invertibility of ($\multimap$R), $A^{[x]}\vdash x$ is provable in $\mathbf{BCK}$ if and only if $\vdash \neg_x A^{[x]}$ is provable in $\mathbf{BCK}$; thus we conclude that $\vdash A$ is provable in the $\{\otimes,\parr\}$-fragment of $\mathbf{LLW}$ if and only if $\vdash \neg_x A^{[x]}$ is provable in $\mathbf{BCK}$.
\end{proof}

In particular, the translation $(\_)^{[\bot]}$ is a sort of standard negative translation from $\mathbf{LLW}$ (resp.\ $\mathbf{ELLW}$) to $\mathbf{ILZW}$ (resp.\ $\mathbf{IEZW}$).
By arguments similar to those in the proof of Theorem~\ref{illtoll5} (cf.\ Appendix~\ref{detailstranslation}), one can easily show the following:
\begin{Thm}\label{illtoll6}
	Let $\vdash \Gamma$ be a classical $\mathcal{L}_C$-sequent. $\vdash \Gamma$ is provable in $\mathbf{LLW}$ (resp.\ $\mathbf{ELLW}$) if and only if $\Gamma^{\jump{\bot}} \vdash$ is provable in $\mathbf{ILZW}$ (resp.\ $\mathbf{IEZW}$). 
\end{Thm}

\section{Alternating branching VASS}
\label{abvass}
The whole content of this section is taken from~\cite[Section~3]{LS15}.
Let $d$ be in $\mathbb{N}$.
The symbols $\vec v_1,\vec v_2,\ldots$ are used to denote $d$-dimensional vectors.
In particular, we write $\vec e_i$ for the \emph{$i$-th unit vector} in $\mathbb{N}^d$ (i.e., the vector with a one in the $i$-th coordinate and zeros elsewhere), and $\vec 0$ for the vector whose every coordinate is zero. 

An \emph{alternating branching vector addition system with states and full zero tests} (\abvass, for short) is a structure of the form $\mathcal{A}=\langle Q,d,T_u,T_s,T_f,T_z \rangle$ where: 
\begin{itemize}
	\item $Q$ is a finite set,
	\item $d$ is in $\mathbb{N}$,
	\item $T_u$ is a finite subset of $Q \times \mathbb{Z}^d \times Q$,
	\item $T_s$ and $T_f$ are subsets of $Q^3$, and
	\item $T_z$ is a subset of $Q^2$. 
\end{itemize}

We call $Q$ a \emph{state space}, $d$ a \emph{dimension}, and $T_u$ (resp.\ $T_s$, $T_f$, $T_z$) the set of \emph{unary} (resp.\ \emph{split}, \emph{fork}, \emph{full zero test}) rules of $\mathcal{A}$. 
For readability, we always write $q \xrightarrow{\vec u} q'$ for $(q,\vec u,q') \in T_u$, $q \rightarrow q_1 \land q_2$
for $(q,q_1,q_2) \in T_f$, $q \rightarrow q_1 + q_2$
for $(q,q_1,q_2) \in T_s$, and $q \fz q'$ for $(q,q') \in T_z$.
A \emph{configuration} of $\mathcal{A}$ is an element of $Q \times \mathbb{N}^d$.

Given an \abvass\ $\mathcal{A}=\langle Q,d,T_u,T_s,T_f,T_z \rangle$, the operational semantics for $\mathcal{A}$ is given by a deduction system over configurations in $Q \times \mathbb{N}^d$. 
It consists of the deduction rules depicted in Figure~\ref{deduction}.
\begin{figure}[t]
\begin{center}
\begin{tabular}{c@{\hspace{4em}}c}
\rootAtTop
\AxiomC{$q',\vec v+\vec u$}
\RightLabel{$(q\xrightarrow{\vec u}q')$}
\UnaryInfC{$q,\vec v$}
\DisplayProof
&
\rootAtTop
\AxiomC{$q',\vec 0$}
\RightLabel{$(q \fz q')$}
\UnaryInfC{$q,\vec 0$}
\DisplayProof
\\[1.5em]
\rootAtTop
\AxiomC{$q_1,\vec v$}
\AxiomC{$q_2,\vec v$}
\RightLabel{$(q \rightarrow q_1 \wedge q_2)$}
\BinaryInfC{$q,\vec v$}
\DisplayProof
&
\rootAtTop
\AxiomC{$q_1,\vec v_1$}
\AxiomC{$q_2,\vec v_2$}
\RightLabel{$(q \rightarrow q_1 + q_2)$}
\BinaryInfC{$q,\vec v_1+\vec v_2$}
\DisplayProof
\end{tabular}
\end{center}
\caption{The deduction rules of an \abvass\ $\mathcal{A}=\langle Q,d,T_u,T_s,T_f,T_z \rangle$, where $q \fz q' \in T_z$, $q \xrightarrow{\vec u}q' \in T_u$, $q \rightarrow q_1 + q_2 \in T_s$, and $q \rightarrow q_1 \wedge q_2 \in T_f$. 
The symbol $+$ stands for componentwise addition and $\vec v+ \vec u$ must be in $\mathbb{N}^d$.}
\label{deduction}
\end{figure}
Given a subset $Q_{\ell}$ of $Q$, a \emph{$Q_{\ell} \times \{\vec 0\}$-leaf-covering deduction tree} in $\mathcal{A}$ is a finite tree labeled by configurations, where leaves are all in $Q_{\ell} \times \{\vec 0\}$ and each other node is obtained from its children by applying one of the deduction rules derived from $T_u \cup T_s \cup T_f \cup T_z$.  
A deduction tree $\mathcal{T}$ whose root configuration is $q,\vec v$ is denoted by the following figure:
\begin{prooftree}
	\AxiomC{$q,\vec v$}
	\noLine
	\UnaryInfC{$\mathcal{T}$}
\end{prooftree}
We write $\mathcal{A},Q_{\ell} \judge q,\vec v$ if there exists a $Q_{\ell} \times \{\vec 0\}$-leaf-covering deduction tree whose root configuration is $q,\vec v$.

In addition, we also give an account of another semantics for \abvass s. 
Given an \abvass\ $\mathcal{A}=\langle Q,d,T_u,T_s,T_f,T_z \rangle$, the \emph{lossy semantics} for $\mathcal{A}$ is given by the aforementioned deduction system of $\mathcal{A}$ augmented with the following additional deduction rules:
\begin{prooftree}
	\AxiomC{$q,\vec v + \vec e_i$}
	\RightLabel{$(\mathsf{loss})$}
	\UnaryInfC{$q,\vec v$}
\end{prooftree}
for every $q \in Q$ and every $i \in \{1,\ldots,d\}$.
In the natural way, we define the notion of \emph{$Q_{\ell} \times \{\vec 0\}$-leaf-covering lossy deduction tree} in $\mathcal{A}$.
We write $\mathcal{A},Q_{\ell} \judge_{\ell} q,\vec v$ if there exists a $Q_{\ell} \times \{\vec 0\}$-leaf-covering lossy deduction tree with root $q,\vec v$. 

\section{Some Tower-complete problems}
\label{tower}
Following the terminology of \cite{LS15,Sch16}, we define
$$
\textsc{Tower}:=\bigcup_{f \in \mathsf{FELEMENTARY}} \textsc{DTime}(\tet{2}{f(n)})
$$
where $\mathsf{FELEMENTARY}$ denotes the set of elementary functions.
This class of problems is closed under elementary many-one reductions (and elementary Turing reductions), i.e., for any two languages $X$ and $Y$, if there is an elementary reduction from $X$ to $Y$ and $Y$ is in \textsc{Tower}, then $X$ is in \textsc{Tower}. 
The notion of \emph{\textsc{Tower}-completeness} is defined with respect to elementary reductions in the usual manner.
For an elaborate discussion on the fast-growing complexity hierarchy, we refer the reader to \cite{Sch16}.

For later use, we summarize here some \tower-complete problems.
Given an \abvass\ $\mathcal{A}=\langle Q,d,T_u,T_s,T_f,T_z \rangle$, $Q_{\ell} \subseteq Q$, and $q_r \in Q$, the \emph{reachability problem} (resp.\ \emph{lossy reachability problem}) asks whether it holds that $\mathcal{A},Q_{\ell} \judge q_r,\vec 0$ (resp.\ $\mathcal{A},Q_{\ell} \judge_{\ell} q_r,\vec 0$). 

\begin{Thm}[Lazi\'c and Schmitz~\cite{LS15}, Theorem~3.6]
	\label{abvass1}
	The lossy reachability problem for \abvass s is \tower-complete.
\end{Thm}

In contrast, the reachability problem is undecidable for \abvass s. 
In fact, the same holds even for alternating VASSs, which are \abvass s with only unary rules and fork rules; see \cite[Section~3.4]{LMSS92} and~\cite[Section~3.3.1]{LS15} for details. 

Intriguingly, the lossy reachability problem is complete for \textsc{Tower} even when a restricted version of \abvass s is considered.
An \emph{ordinary \abvass}\ is an \abvass\ $\mathcal{A}=\langle Q,d,T_u,T_s,T_f,T_z \rangle$ where for every $q \xrightarrow{\vec u}q' \in T_u$ it holds that either $\vec u=\vec e_i$ or $\vec u=-\vec e_i$ for some
$1 \leq i \leq d$. 
A \emph{branching vector addition system with states} (BVASS, for short) is an \abvass\ $\mathcal{A}=\langle Q,d,T_u,T_s,T_f,T_z \rangle$ where $T_f$ and $T_z$ are both empty. 
The main results in the remaining sections rely crucially on the following two complexity results:

\begin{Thm}[Lazi\'c and Schmitz~\cite{LS15}, Lemma~3.5 and Theorem~3.6]
	\label{bvass1}
	The lossy reachability problem for ordinary BVASSs is \textsc{Tower}-complete.
\end{Thm}

\begin{Thm}[Lazi\'c and Schmitz~\cite{LS15}; Fact~4.2, Corollary~5.4 and Corollary 6.3]
	\label{logic}
	The provability problems for $\mathbf{ILZW}$ and $\mathbf{LLW}$ are \textsc{Tower}-complete. 
\end{Thm}

\section{Membership in Tower of contraction-free logics}\label{encodingoflogic}
This section consists of two parts. 
We first show that provability in elementary affine logic is in \textsc{Tower} (Section~\ref{towerupperellw}). 
Secondly, we also show the \tower-membership of deducibility in $\mathbf{FL}_{\mathbf{ew}}$ and related logical systems (Section~\ref{towerupperflew}). 

\subsection{Tower upper bound for provability in elementary affine logic}\label{towerupperellw}
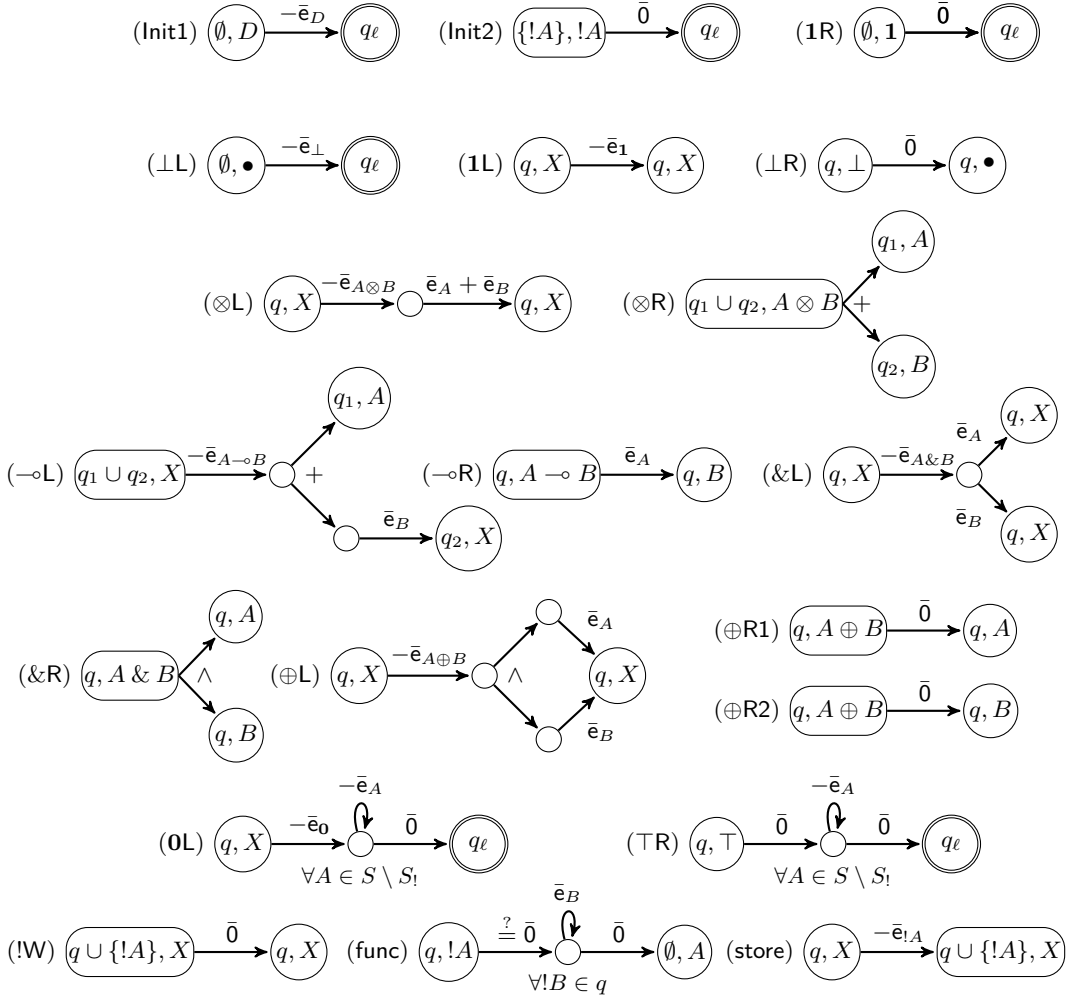
\begin{figure}[t]
	\centering
	\begin{tikzpicture}[auto,draw,node distance=1.0cm]
	\node[state,label=left:{($\mathsf{Init1}$)}](Init11){$\emptyset,D$};
	\node[state,accepting by double,right=of Init11,inner sep=4pt](Init12){$q_\ell$};
	\node[rstate,label=left:{($\mathsf{Init2}$)},right=1.5cm of Init12](Init13){$\{\oc A\},\oc A$};
	\node[state,accepting by double,right=of Init13,inner sep=4pt](Init14){$q_\ell$};   
	\node[state,right=1.5cm of Init14,label=left:{($\mathbf{1}\mathsf{R}$})](OneR1){$\emptyset,\mathbf 1$};
	\node[state,accepting by double,right=of OneR1,inner sep=4pt](OneR2){$q_\ell$};
	\path[->, >=stealth] (OneR1) edge node{$\vec 0$} (OneR2);
	\path[->, >=stealth] (Init11) edge node{$-\vec e_{D}$} (Init12)
	(Init13) edge node{$\vec 0$} (Init14)
	(OneR1) edge node{$\vec 0$} (OneR2);
	\node[state,below=1cm of
	Init11,inner sep=2.0pt,label=left:{($\bot\mathsf{L}$)}](BottomL1){$\emptyset,\bullet$};
	\node[state,accepting by double,right=of BottomL1,inner sep=4pt](BottomL2){$q_\ell$};
	\path[->, >=stealth] (BottomL1) edge node{$-\vec e_\bot$} (BottomL2);
	\node[state, right= 1.5cm of BottomL2,label=left:{($\mathbf{1}\mathsf{L}$)}](OneL1){$q,X$};
	\node[state,right=of OneL1](OneL2){$q,X$};
	\path[->, >=stealth] (OneL1) edge node{$-\vec e_{\mathbf 1}$} (OneL2);
	\node[state,right=1.5cm of OneL2,label=left:{($\bot\mathsf{R}$})](BottomR1){$q,\bot$};
	\node[state,inner sep=2.5pt,right=of BottomR1](BottomR2){$q,\bullet$};
	\path[->, >=stealth] (BottomR1) edge node{$\vec 0$} (BottomR2);
	\node[state,below right=1.3cm and 0.2cm of BottomL1,label=left:{($\otimes\mathsf{L}$)}](TensorL1){$q,X$};
	\node[intermediate,right=1cm of TensorL1](TensorL2){};
	\node[state,right=1.2cm of TensorL2](TensorL3){$q,X$};
	\path[->, >=stealth] (TensorL1) edge node{$-\vec e_{A\otimes B}$} (TensorL2)
	(TensorL2) edge node{$\vec e_A+\vec e_B$} (TensorL3);
	\node[rstate,right=1.5cm of TensorL3,label=left:{($\otimes\mathsf{R}$)}, label=right:{$+$}](TensorR1){$q_1\cup q_2,A\otimes B$};
	\node[state,above right=.2cm and .5cm of TensorR1](TensorR2){$q_1,A$};
	\node[state,below right=.2cm and .5cm of TensorR1](TensorR3){$q_2,B$};
	\path[->, >=stealth] (TensorR1.east) edge (TensorR2)
	(TensorR1.east) edge (TensorR3);
    %
	\node[rstate,below left=3.5cm and .4cm of BottomL1,label=left:{($\multimap$\textsf{L})}](ImplicationL1){$q_1\cup q_2,X$};
	\node[intermediate,right=1.1cm of ImplicationL1,label=right:{$+$}](ImplicationL2){};
	\node[state,above right=.6cm and .6cm of ImplicationL2](ImplicationL3){$q_1,A$};
	\node[intermediate,below right=.6cm and .6cm of ImplicationL2](ImplicationL4){};
	\node[state,right=of ImplicationL4](ImplicationL5){$q_2,X$};
	\path[->, >=stealth] (ImplicationL1) edge node{$-\vec e_{A\multimap B}$} (ImplicationL2)
	(ImplicationL2) edge (ImplicationL3)
	(ImplicationL2) edge (ImplicationL4)
	(ImplicationL4) edge node{$\vec e_B$}(ImplicationL5);
	\node[rstate,right=2.6cm of
	ImplicationL2,label=left:{($\multimap$\textsf{R})}](ImplicationR1){$q,A\multimap B$};
	\node[state,right=of ImplicationR1](ImplicationR2){$q,B$};
	\path[->, >=stealth] (ImplicationR1) edge node{$\vec e_A$} (ImplicationR2);
	\node[state,right=1.2cm of
	ImplicationR2,label=left:{($\with\mathsf{L}$)}](WithL1){$q,X$};
	\node[intermediate,right=1cm of WithL1](WithL2){};
	\node[state,above right=.4 cm and .4cm of WithL2](WithL3){$q,X$};
	\node[state,below right=.4cm and .4cm of WithL2](WithL4){$q,X$};
	\path[->, >=stealth] (WithL1) edge node{$-\vec e_{A\with B}$} (WithL2)
	(WithL2) edge node{$\vec e_A$} (WithL3)
	(WithL2) edge[swap] node{$\vec e_B$} (WithL4);
	\node[rstate,below= 2cm of
	ImplicationL1,label=left:{($\with\mathsf{R}$)},label=right:{$\,\,\wedge$}](WithR1){$q,A\with B$};
	\node[state,above right=.2cm and .5cm of WithR1](WithR2){$q,A$};
	\node[state,below right=.2cm and .5cm of WithR1](WithR3){$q,B$};
	\path[->, >=stealth] (WithR1.east) edge (WithR2)
	(WithR1.east) edge (WithR3);
	\node[state,right=2cm of WithR1,label=left:{($\oplus\mathsf{L}$)}](PlusL1){$q,X$};
	\node[intermediate,right=1.1cm of PlusL1,label=right:{$\wedge$}](PlusL2){};
	\node[intermediate,above right= .6cm and .6cm of PlusL2](PlusL3){};
	\node[intermediate,below right= .6cm and .6cm of PlusL2](PlusL4){};
	\node[state,right=1.2cm of PlusL2](PlusL5){$q,X$};
	\path[->, >=stealth] (PlusL1) edge node{$-\vec e_{A\oplus B}$} (PlusL2)
	(PlusL2) edge (PlusL3)
	(PlusL2) edge (PlusL4)
	(PlusL3) edge node{$\vec e_A$} (PlusL5)
	(PlusL4) edge[swap] node{$\vec e_B$} (PlusL5);
	\node[rstate,above right= .0cm and 2cm of
	PlusL5,label=left:{($\oplus\mathsf{R1}$)}](PlusR11){$q,A\oplus B$};
	\node[state,right=of PlusR11](PlusR12){$q,A$};
	\path[->, >=stealth] (PlusR11.east) edge node {$\vec 0$} (PlusR12);
	\node[rstate,below=0.4cm of
	PlusR11,label=left:{($\oplus\mathsf{R2}$)}](PlusR21){$q,A\oplus B$};
	\node[state,right=of PlusR21](PlusR22){$q,B$};
	\path[->, >=stealth] (PlusR21.east) edge node {$\vec 0$} (PlusR22);
	\node[state,below left =1.7cm and 1cm of PlusL1,label=left:{($\mathbf{0}\mathsf{L}$)}](ZeroL1){$q,X$};
	\node[intermediate,right=of ZeroL1,label=below:{{\footnotesize $\forall A\in
			S\setminus S_\oc$}}](ZeroL2){};
	\node[state,accepting by double,right=of ZeroL2,inner sep=4pt](ZeroL3){$q_\ell$};
	\path[->, >=stealth] (ZeroL1) edge node{$-\vec e_{\mathbf{0}}$} (ZeroL2)
	(ZeroL2) edge[loop above] node {$-\vec e_A$} ()
	(ZeroL2) edge node{$\vec 0$} (ZeroL3);
	\node[state,right=5.5cm of ZeroL1,label=left:{($\top\mathsf{R}$)}](TopR1){$q,\top$};
	\node[intermediate,right=of TopR1,label=below:{{\footnotesize $\forall A\in
			S\setminus S_\oc$}}](TopR2){};
	\node[state,accepting by double,right=of TopR2,inner sep=4pt](TopR3){$q_\ell$};
	\path[->, >=stealth] (TopR1) edge node{$\vec 0$} (TopR2)
	(TopR2) edge[loop above] node {$-\vec e_A$} ()
	(TopR2) edge node{$\vec 0$} (TopR3);
	\node[rstate,below= 3.0cm of WithR1,label=left:{($\oc\mathsf{W}$)}](Weak1){$q\cup\{\oc A\},X$};
	\node[state,right=of Weak1](Weak2){$q,X$};
	\path[->, >=stealth] (Weak1) edge node{$\vec 0$} (Weak2);
	\node[state,right= 1.2cm of Weak2,label=left:{($\mathsf{func}$)}](F1){$q,\oc A$};
	\node[intermediate,right=of F1,label=below:{{\footnotesize $\forall \oc B \in
			q$}}](F2){};
	\node[state,right=of F2](F3){$\emptyset,A$};
	\path[->, >=stealth] (F1) edge node{$\stackrel{?}{=}\vec 0$} (F2)
	(F2) edge[loop above] node {$\vec e_B$} ()
	(F2) edge node{$\vec 0$} (F3);
	\node[state,right=1.2cm of F3,label=left:{($\mathsf{store}$)}](s1){$q,X$};
	\node[rstate,right=of s1](s2){$q \cup \{\oc A\},X$};
	\path[->, >=stealth] (s1) edge node{$-\vec e_{\oc A}$} (s2);
	\end{tikzpicture}
	\caption{\label{rule3}Rules and intermediate states of $\mathcal{A}^E_F$.
	All formulas range over $S$, $q,q_1,q_2$ range over $\mathcal{P}(S_{\oc})$, $D$ ranges over $S \setminus S_{\oc}$, and $X$ ranges over $S \cup \{\bullet\}$. Small circles stand for intermediate states.}
\end{figure}
We show that there exists an elementary reduction from the provability problem for $\mathbf{IEZW}$ to the lossy reachablity problem for \abvass s. 
Our reduction is a slightly modified version of the (polynomial-space) reduction, given in \cite[Section 4.1.2]{LS15}, from provability in $\mathbf{ILZW}$ to lossy reachability in \abvass.

Let $F$ be an intuitionistic $\mathcal{L}$-formula. 
Let $S$ be the set of subformulas of $F$, $S_{\oc}$ the set of formulas in $S$ of the form $\oc B$, and $\bullet$ a fresh symbol not in $S$. 
For $\Pi \in S \cup \{\varepsilon\}$, we define $\Pi^{\dagger}=A$ if $\Pi=A$, and $\Pi^{\dagger}=\bullet$ otherwise. 
A multiset over $S$ is just a map from $S$ to $\mathbb{N}$, i.e., an element of $\mathbb{N}^S$. 
Given a multiset $m$ over $S$, $m(B)$ denotes the multiplicity of a formula $B$ in $m$. 

Fix an enumeration $F_1,\ldots,F_{d}$ of all the formulas in $S$. 
A multiset $m$ over $S$ can be expressed as $F_1^{m(F_1)},\ldots,F_{d}^{m(F_{d})}$. 
For each multiset $m$ over $S$, we write $\vec v_{m}$ for the vector $\langle m(F_1),\ldots,m(F_{d})\rangle$ in $\mathbb{N}^{d}$. 
In particular, we write $\vec e_B$ for the vector $\vec v_B$ corresponding to a formula $B$ in $S$.
Note that $\vec v_{m}=\vec 0$ if $m$ is the empty multiset.
We write $\sigma(m)$ for the \emph{support} of a multiset $m$, i.e., $\{B \in S \mid m(B)>0\}$. 
We then construct an \abvass\ $\mathcal{A}^E_{F}$ as below: 
\begin{itemize}
	\item The dimension of $\mathcal{A}^E_F$ is $d$ ($=|S|$).
	\item The state space of $\mathcal{A}^E_F$ contains $\mathcal{P}(S_{\oc}) \times (S \cup \{\bullet\})$, a distinguished leaf state $q_{\ell}$, and several intermediate states which are needed for defining the rules for ($\multimap\!\!\mathsf{L}$), ($\with\mathsf{L}$), ($\oplus\mathsf{L}$), ($\otimes\mathsf{L}$), ($\mathbf{0}\mathsf{L}$), ($\top\mathsf{R}$) and ($\mathsf{func}$) in $\mathcal{A}^E_F$ (cf.\ Figure~\ref{rule3}).
	\item The rules and intermediate states of $\mathcal{A}^E_F$ are defined as in Figure \ref{rule3}. 
\end{itemize}

The construction of $\mathcal{A}^E_F$ is quite similar to that of the \abvass\ $\mathcal{A}^I_F$ defined in \cite[Section 4.1.2]{LS15}. 
We stress that $\mathcal{A}^E_F$ has the rules for ($\mathsf{func}$) instead of the rules for ($\oc \mathsf{D}$) and ($\mathsf{\oc P}$) (see Figure~\ref{threerules} in Section~\ref{towerupperflew}), whereas $\mathcal{A}^I_F$ has the rules for ($\oc \mathsf{D}$) and ($\mathsf{\oc P}$) instead of the rules for ($\mathsf{func}$). 
Notice that $\mathcal{A}^E_F$ does not have the rules corresponding to the inference rule of (W) in $\mathbf{ILZW}$. 
The left-weakening rule is implemented by loss rules derived from the lossy semantics for $\mathcal{A}^E_F$.

Let $\Theta,\Gamma \vdash \Pi$ be an intuitionistic $\mathcal{L}$-sequent such that $\Theta$ is a multiset of formulas in $S_{\oc}$, $\Gamma$ is a multiset of formulas in $S\setminus S_{\oc}$, and $\Pi$ is in $S \cup \{\varepsilon\}$. 
It is translated as the configuration $\sigma(\Theta),\Pi^{\dagger},\vec v_{\Gamma}$ in $\mathcal{P}(S_{\oc}) \times (S \cup \{\bullet\}) \times \mathbb{N}^d$. 
We now show the key theorem of this subsection; see Appendix~\ref{appmembertower1} for a detailed proof.

\begin{restatable}{Thm}{iezwtoabvass}\label{iezwtoabvass}
	Let $\Pi$ be in $S\cup \{\varepsilon\}$, $\Theta$ a multiset of formulas in $S_{\oc}$, $\Gamma$ a multiset of formulas in $S\setminus S_{\oc}$. 
	$\Theta,\Gamma \vdash \Pi$ is provable in $\mathbf{IEZW}$ if and only if $\mathcal{A}^{E}_F,\{q_{\ell}\} \judge_{\ell} \sigma(\Theta),\Pi^{\dagger},\bar{\mathsf{v}}_{\Gamma}$.
\end{restatable}

In particular, Theorem \ref{iezwtoabvass} guarantees that for any intuitionistic $\mathcal{L}$-formula $F$, $F$ is provable in $\mathbf{IEZW}$ if and only if $\mathcal{A}^{E}_F,\{q_{\ell}\} \judge_{\ell} \emptyset,F,\vec 0$. 
By Theorems~\ref{abvass1} and~\ref{illtoll6}:

\begin{Cor}
	\label{towermembership}
	The provability problems for $\mathbf{IEZW}$ and $\mathbf{ELLW}$ are in \textsc{Tower}.
\end{Cor}

We stress that the \textsc{Tower} upper bound also holds for fragments of $\mathbf{IEZW}$ and $\mathbf{ELLW}$, e.g., the $\{\multimap,\oc\}$-fragment of $\mathbf{IEZW}$. 

\subsection{Tower upper bound for deducibility in FLew and related systems}\label{towerupperflew}
We first describe the notion of \emph{$\oc$-prenex sequent} which originates in~\cite[Section~2]{Ter02}. 
Let $\mathcal{K}$ be a language such that $\bot \in \mathcal{K} \subseteq \mathcal{L}$.
A \emph{$\oc$-prenex $\mathcal{K}$-sequent} is an intuitionistic sequent of the form $\oc \Gamma,\Delta \vdash \Pi$, where $\Gamma$ and $\Delta$ are finite multisets of intuitionistic $\mathcal{K}$-formulas, and $\Pi$ is an intuitionistic $\mathcal{K}$-formula or the empty multiset. 
Similarly, let $\mathcal{K}$ be a sublanguage of $\mathcal{L}_C$. 
A \emph{$\wn$-prenex $\mathcal{K}$-sequent} is a right-hand sided sequent of the form $\vdash \wn \Gamma,\Delta$, where $\Gamma$ and $\Delta$ are finite multisets of classical $\mathcal{K}$-formulas. 
We write $\Gamma^n$ for the multiset sum of $n$ copies of $\Gamma$ for each $n \geq 0$. 

\begin{Lem}\label{ilzwtoflei}
	Let $\oc \Gamma,\Delta \vdash \Pi$ be a $\oc$-prenex $\{\otimes,\multimap,\with,\oplus,\mathbf{1},\bot\}$-sequent. 
	\begin{bracketenumerate}
		\item If $\oc \Gamma,\Delta \vdash \Pi$ is provable in $\mathbf{ILZW}$, then $\Gamma^n,\Delta \vdash \Pi$ is provable in $\mathbf{FL}_{\mathbf{ei}}$ for some $n \geq 0$.
		\item If $\oc \Gamma,\Delta \vdash \Pi$ is provable in $\mathbf{ILZW}'$, then $\Gamma^n,\Delta \vdash \Pi$ is provable in $\mathbf{FL}_{\mathbf{ew}}$ for some $n \geq 0$.
	\end{bracketenumerate}
\end{Lem}
\begin{proof}
    The proof of Statement (1) proceeds by induction on the size of the cut-free proof of $\oc \Gamma,\Delta \vdash \Pi$ in $\mathbf{ILZW}$. 
	We perform a case analysis, depending on which inference rule is applied last.
	
	We consider only the case of ($\multimap$L).
	If $\oc \Gamma,\oc \Sigma,A \multimap B,\Delta,\Xi\vdash \Pi$ is obtained from $\oc \Gamma,\Delta \vdash A$ and $\oc \Sigma,B,\Xi \vdash \Pi$ by an application of ($\multimap$L), then by the induction hypothesis, $\Gamma^{n'},\Delta \vdash A$ is provable in $\mathbf{FL}_{\mathbf{ei}}$ for some $n'$, and $\Sigma^{n''},B,\Xi \vdash \Pi$ is provable in $\mathbf{FL}_{\mathbf{ei}}$ for some $n''$. 
	Applying ($\multimap$L) we obtain a proof of $\Gamma^{n'},\Sigma^{n''},A \multimap B,\Delta,\Xi\vdash \Pi$ in $\mathbf{FL}_{\mathbf{ei}}$. 
	Note that $n'$ is not always equal to $n''$.
    However, we may unify them using the rule of (W); thus $(\Gamma,\Sigma)^{n'+n''},A \multimap B,\Delta,\Xi\vdash \Pi$ is provable in $\mathbf{FL}_{\mathbf{ei}}$. 
	The remaining cases are similar.
 
    We can show Statement (2) similarly.
\end{proof}

Similarly to the above theorem, one can also show the following:
\begin{Lem}\label{llwtoinflei}
	Let $\vdash \wn \Gamma,\Delta$ be a $\wn$-prenex $\{\otimes,\parr,\with,\oplus,\mathbf{1},\bot\}$-sequent. 
	If $\vdash \wn \Gamma,\Delta$ is provable in $\mathbf{LLW}$, then $\vdash \Gamma^n,\Delta$ is provable in $\mathbf{InFL}_{\mathbf{ew}}$ for some $n \geq 0$.
\end{Lem}

Lemmas~\ref{ilzwtoflei}~and~\ref{llwtoinflei} are affine analogues of \cite[Proposition~2.6]{Ter02}. 
Interestingly, a similar idea is found in the proof of the \emph{local deduction theorem} for $\mathbf{FL}_{\mathbf{ew}}$; see \cite[Corollary~2.15]{GO06}. 
The following lemma, which is inspired by~\cite[Lemmas 3.2 and 3.3]{LMSS92}, provides a very simple reduction from $\mathbf{FL}_{\mathbf{ei}}$ deducibility (resp.\ $\mathbf{FL}_{\mathbf{ew}}$ deducibility) into $\mathbf{ILZW}$ provability (resp.\ $\mathbf{ILZW}'$ provability). 

\begin{Lem}\label{fleitoilzw}
	Let $\Phi$ be a finite set of intuitionistic $\{\otimes,\multimap,\with,\oplus,\mathbf{1},\bot\}$-formulas and $\Gamma \vdash \Pi$ an intuitionistic $\{\otimes,\multimap,\with,\oplus,\mathbf{1},\bot\}$-sequent. 
	\begin{bracketenumerate}
		\item $\Gamma \vdash \Pi$ is provable in $\mathbf{FL_{\mathbf{ei}}}[\Phi]$ if and only if $\oc \Phi,\Gamma \vdash \Pi$ is provable in $\mathbf{ILZW}$.
		\item $\Gamma \vdash \Pi$ is provable in $\mathbf{FL}_{\mathbf{ew}}[\Phi]$ if and only if $\oc \Phi,\Gamma \vdash \Pi$ is provable in $\mathbf{ILZW}'$.
	\end{bracketenumerate}
\end{Lem}
\begin{proof}
	We only prove the first statement, the proof of the second one being similar. 
	For the \emph{if} part, let us suppose that $\oc \Phi,\Gamma \vdash \Pi$ is provable in $\mathbf{ILZW}$.
	Since $\oc \Phi,\Gamma \vdash \Pi$ is a $\oc$-prenex $\{\otimes,\multimap,\with,\oplus,\mathbf{1},\bot\}$-sequent, $\Phi^n,\Gamma \vdash \Pi$ is provable in $\mathbf{FL_{\mathbf{ei}}}$ for some $n$, by Lemma~\ref{ilzwtoflei}. 
	Since $\vdash B$ is an initial sequent of $\mathbf{FL_{\mathbf{ei}}}[\Phi]$ for each $B \in \Phi$, we can construct a proof of $\Gamma \vdash \Pi$ in $\mathbf{FL_{\mathbf{ei}}}[\Phi]$ by several applications of (Cut).
	The \emph{only-if} part follows by induction on the height of the proof of $\Gamma \vdash \Pi$ in $\mathbf{FL}_{\mathbf{ei}}[\Phi]$. 
\end{proof}

In the same way as before, the deducibility problems for $\mathbf{FL}^+_{\mathbf{ei}}$ and $\mathbf{BCK}$ are also reduced to the provability problem for $\mathbf{ILZW}$. 
Furthermore, one can show that there is also a straightforward reduction from $\mathbf{InFL}_{\mathbf{ew}}$ deducibility to $\mathbf{LLW}$ provability, using Lemma~\ref{llwtoinflei}:

\begin{Lem}\label{infleitollw}
	Let $\Phi$ be a finite set of classical $\{\otimes,\parr,\with,\oplus,\mathbf{1},\bot\}$-formulas and $\vdash \Gamma$ a classical $\{\otimes,\parr,\with,\oplus,\mathbf{1},\bot\}$-sequent.
	$\vdash\Gamma$ is provable in $\mathbf{InFL}_{\mathbf{ew}}[\Phi]$ if and only if $\vdash \wn \Phi^{\bot},\Gamma$ is provable in $\mathbf{LLW}$. 
\end{Lem}

Recall that the provability problems for $\mathbf{ILZW}$ and $\mathbf{LLW}$ are both in \textsc{Tower} by Theorem~\ref{logic}. 
We obtain:

\begin{Cor}
	\label{towerdeducibility}
	The following decision problems are in \ts{Tower}.
	\begin{itemize}
	    \item the deducibility problem for $\mathbf{BCK}$, 
	    \item the deducibility problem for $\mathbf{FL}^+_{\mathbf{ei}}$, 
	    \item the deducibility problem for $\mathbf{FL}_{\mathbf{ei}}$,
	    \item the deducibility problem for $\mathbf{InFL}_{\mathbf{ew}}$.
	\end{itemize}
\end{Cor}

At the end of this section, we show the membership in \textsc{Tower} of the deducibility problem for $\mathbf{FL}_{\mathbf{ew}}$. 
It suffices by Lemma~\ref{fleitoilzw} to show that the provability problem for $\mathbf{ILZW}'$ is in \textsc{Tower}.

Let $F$ be an instance of the provability problem for $\mathbf{ILZW}'$.
As before, $S$ denotes the set of subformulas of $F$, $S_{\oc}$ the set of formulas in $S$ of the form $\oc B$, and $\bullet$ a distinguished symbol. 
We then construct an \abvass\ $\mathcal{A}^{I'}_F$ of dimension $|S|$, by modifying the construction of $\mathcal{A}^E_F$ given in the previous subsection.
The state space of $\mathcal{A}^{I'}_F$ contains $\mathcal{P}(S_{\oc}) \times (S \cup \{\bullet\})$, a distinguished leaf state $q_{\ell}$, and intermediate states that are needed for defining the rules for ($\multimap$\textsf{L}), ($\with\mathsf{L}$), ($\oplus\mathsf{L}$), ($\otimes\mathsf{L}$), ($\mathbf{0}\mathsf{L}$), and ($\top\mathsf{R}$), cf.\ Figure~\ref{rule3}. 
The rules of $\mathcal{A}^{I'}_F$ are ($\mathsf{Init1}$), ($\mathsf{Init2}$), ($\mathsf{store}$), ($\mathbf{1}\mathsf{L}$),  ($\mathbf{1}\mathsf{R}$), ($\bot\mathsf{L}$), ($\bot\mathsf{R}$), ($\multimap$\textsf{L}), ($\multimap$\textsf{R}), ($\otimes\mathsf{L}$), ($\otimes\mathsf{R}$), ($\with\mathsf{L}$), ($\with\mathsf{R}$), ($\oplus\mathsf{L}$), ($\oplus\mathsf{R1}$), ($\oplus\mathsf{R2}$), ($\mathbf{0}\mathsf{L}$), ($\top\mathsf{R}$), ($\oc\mathsf{W}$), 
($\mathsf{W'}$), ($\oc\mathsf{D}$) and ($\oc\mathsf{P}$), all of which except for ($\mathsf{W'}$), ($\oc\mathsf{D}$) and ($\oc\mathsf{P}$), are depicted in Figure~\ref{rule3}. 
The rules for ($\mathsf{W'}$), ($\oc\mathsf{D}$) and ($\oc\mathsf{P}$) are defined as in Figure~\ref{threerules}.
\begin{figure}[t]
	\begin{center}
	\begin{tikzpicture}[auto,draw,node distance=1.0cm]
		\node[state,label=left:{($\mathsf{W}'$)}](rw1){$q,A$};
	\node[state,right=of rw1,inner sep=2.5pt](rw2){$q,\bullet$};   
	\path[->, >=stealth] (rw1) edge node{$\vec 0$} (rw2);
	\node[rstate,right=1.5cm of rw2,label=left:{($\oc\mathsf{D}$)}](d1){$q\cup\{\oc A\},X$};
	\node[state,right=of d1](d2){$q,X$};
	\path[->, >=stealth] (d1) edge node{$\vec e_{A}$} (d2);
	\node[state,right=1.5cm of d2,label=left:{($\oc\mathsf{P}$)}](p1){$q,\oc A$};
	\node[state,right=of p1](p2){$q,A$};
	\path[->, >=stealth] (p1) edge node{$\stackrel{?}{=}\vec 0$} (p2);
	\end{tikzpicture}
	\end{center}
	\caption{The \abvass\ rules for ($\mathsf{W'}$), ($\oc\mathsf{D}$) and ($\oc$\textsf{P})}
	\label{threerules}
	\end{figure}
Clearly, these three types of \abvass\ rules correspond to the inference rules of (W'), ($\oc$D) and ($\oc$P) in $\mathbf{ILZW}'$, respectively. 
Note that $\mathcal{A}^{I'}_F$ is not equipped with the rule for ($\mathsf{func}$). 
We can show that $\mathcal{A}^{I'}_F$ simulates the proof search of $F$ in $\mathbf{ILZW}'$ with the lossy semantics; see Appendix~\ref{appmembertower1} for a proof.

\begin{restatable}{Thm}{encodeILZWprime}\label{encodingofILZW2}
	Let $\Pi$ be in $S\cup \{\varepsilon\}$, $\Theta$ a multiset of formulas in $S_{\oc}$, $\Gamma$ a multiset of formulas in $S\setminus S_{\oc}$. $\Theta,\Gamma \vdash \Pi$ is provable in $\mathbf{ILZW}'$ if and only if $\mathcal{A}^{I'}_F,\{q_{\ell}\} \judge_{\ell} \sigma(\Theta),\Pi^{\dagger},\bar{\mathsf{v}}_{\Gamma}$.
\end{restatable}
Specifically, $\mathcal{A}^{I'}_F,\{q_{\ell}\} \judge_{\ell} \emptyset,F,\vec 0$ if and only if $\vdash F$ is provable in $\mathbf{ILZW}'$; thus Lemma~\ref{fleitoilzw} and Theorem~\ref{abvass1} provide:
\begin{Cor}
	The provability problem for $\mathbf{ILZW'}$ is in \textsc{Tower}.
\end{Cor}

\begin{Cor}
	\label{flewisintower}
	The deducibility problem for $\mathbf{FL}_{\mathbf{ew}}$ is in \textsc{Tower}.
\end{Cor}

\section{Tower-hardness of contraction-free logics}
\label{encodingofbvass}

For our purposes, we recall from \cite[Section 4.2]{LS15} a log-space reduction from the lossy reachability problem for ordinary BVASSs to the provability problem of $\wn$-prenex $\{\otimes,\parr\}$-sequents in $\mathbf{LLW}$. 

Let $(\mathcal{B},Q_{\ell},q_r)$ be an instance of the lossy reachability problem for ordinary BVASSs, where $\mathcal{B}=\langle Q,d,T_u,T_s,\emptyset,\emptyset \rangle$. 
We fix a set $Q \cup \{e_i \mid 1 \leq i \leq d\}$ of propositional variables.
Given $(q,\vec v) \in Q \times \mathbb{N}^{d}$, we define $\theta(q,\vec v)=q^{\bot},(e_1^{\bot})^{\vec v(1)},\ldots,(e_d^{\bot})^{\vec v(d)}$, where $\vec v(i)$ stands for the $i$-th coordinate of $\vec v$.
We write $T$ for the set of the three types of non-logical axioms, each of which is derived from $T_u \cup T_s$ as follows:
\begin{itemize}
	\item $\vdash q^{\bot},q_1 \otimes e_i$ for $q\xrightarrow{\vec e_i}q_1 \in T_u$,
	\item $\vdash q^{\bot},e_i^{\bot},q_1$ for $q\xrightarrow{-\vec e_i}q_1 \in T_u$,
	\item $\vdash q^{\bot},q_1 \parr q_2$ for $q \rightarrow q_1+q_2 \in T_s$.
\end{itemize}
Each sequent in $T$ is of the form $\vdash q_1^{\bot},\ldots,q_n^{\bot},C$,  where $q^{\bot}_1,\ldots,q^{\bot}_n$ are negative literals and $C$ is a classical $\{\otimes,\parr\}$-formula. 
For any sequent $t=\ \vdash q_1^{\bot},\ldots,q_n^{\bot},C$ in $T$, we define $\ulcorner t \urcorner= q_1 \otimes \cdots \otimes q_n \otimes C^{\bot}$.
Given a finite set $T=\{t_1,\ldots,t_n\}$ of sequents, $\ulcorner T \urcorner$ denotes the multiset $\ulcorner t_1 \urcorner,\ldots, \ulcorner t_n \urcorner$. 
It then holds that, for any $(q,\vec v) \in Q \times \mathbb{N}^d$, $\mathcal{B},Q_{\ell} \judge_{\ell} q,\vec v$ if and only if $\vdash \wn \ulcorner T \urcorner,\wn Q_{\ell},\theta(q,\vec v)$ is provable in $\mathbf{LLW}$. 
In particular, the following holds:

\begin{Thm}[Lazi\'c and Schmitz~\cite{LS15}, Section~4.2.3]
	\label{BVASStoLL}
	$\mathcal{B},Q_{\ell} \judge_{\ell} q_r,\vec 0$ if and only if $\vdash \wn \ulcorner T \urcorner,\wn Q_{\ell},q_r^{\bot}$ is provable in $\mathbf{LLW}$. 
\end{Thm}

The key observation here is that $\vdash \wn \ulcorner T \urcorner,\wn Q_{\ell},q_r^{\bot}$ forms a $\wn$-prenex $\{\otimes,\parr\}$-sequent. 
Thus in conjunction with Theorem~\ref{bvass1}, we obtain:

\begin{Cor}[Lazi\'c and Schmitz \cite{LS15}, Section~4.2.3]
	\label{llwistowerhard}
	The problem of determining if a given $\wn$-prenex $\{\otimes,\parr\}$-sequent is provable in $\mathbf{LLW}$ is \textsc{Tower}-hard.
\end{Cor}

\subsection{Tower-hardness of deducibility in FLew and related systems}

A \emph{$\oc$-prenex $\{\multimap\}$-sequent} is an intuitionistic $\{\multimap,\oc\}$-sequent of the form $\oc \Gamma,\Delta \vdash C$ where the only connective occurring in $\Gamma,\Delta,C$ is $\multimap$. 
We prove: 

\begin{Thm}\label{illwistowerhard}
	The problem of determining if a given $\oc$-prenex $\{\multimap\}$-sequent is provable in the $\{\multimap,\oc\}$-fragment of $\mathbf{ILLW}$ is \textsc{Tower}-hard.
\end{Thm}
\begin{proof}
	In view of Corollary~\ref{llwistowerhard}, we reduce from the problem of whether a given $\wn$-prenex $\{\otimes,\parr\}$-sequent is provable in $\mathbf{LLW}$. 
	Let $\vdash \wn \Gamma,\Delta$ be a $\wn$-prenex $\{\otimes,\parr\}$-sequent and $x$ a new propositional variable not occurring in $\wn\Gamma,\Delta$. 
	Theorem~\ref{illtoll5} guarantees that $\vdash \wn \Gamma,\Delta$ is provable in $\mathbf{LLW}$ if and only if $(\wn \Gamma,\Delta)^{[x]} \vdash x$ is provable in $\mathbf{ILLW}$. 
	Clearly, the latter sequent forms a $\oc$-prenex $\{\multimap\}$-sequent. 
    Due to the fact that $\mathbf{ILLW}$ admits cut-elimination, $\mathbf{ILLW}$ is a conservative extension of its $\{\multimap,\oc\}$-fragment; hence $(\wn \Gamma,\Delta)^{[x]} \vdash x$ is provable in $\mathbf{ILLW}$ if and only if it is provable in the $\{\multimap,\oc\}$-fragment of $\mathbf{ILLW}$.
    We conclude that $\vdash \wn \Gamma,\Delta$ is provable in $\mathbf{LLW}$ if and only if $(\wn \Gamma,\Delta)^{[x]} \vdash x$ is provable in the $\{\multimap,\oc\}$-fragment of $\mathbf{ILLW}$. 
    Hence the problem is hard for \tower.
\end{proof}

We furthermore show the following lemma; see Appendix~\ref{appenproofhard} for a proof.
\begin{restatable}{Lem}{prenextransii}\label{prenextrans2}
	Let $\oc \Gamma,\Delta \vdash A$ be a $\oc$-prenex $\{\multimap\}$-sequent. 
	The following statements are mutually equivalent:
	\begin{bracketenumerate}
		\item $\oc \Gamma,\Delta \vdash A$ is provable in the $\{\multimap,\oc\}$-fragment of $\mathbf{ILLW}$,
		\item $\Delta \vdash A$ is provable in $\mathbf{BCK}[\sigma(\Gamma)]$,
		\item $\Delta \vdash A$ is provable in $\mathbf{FL}^+_{\mathbf{ei}}[\sigma(\Gamma)]$,
		\item $\Delta \vdash A$ is provable in $\mathbf{FL}_{\mathbf{ei}}[\sigma(\Gamma)]$,
		\item $\Delta \vdash A$ is provable in $\mathbf{FL}_{\mathbf{ew}}[\sigma(\Gamma)]$.
	\end{bracketenumerate}
Here $\sigma(\Gamma)$ stands for the support of $\Gamma$, i.e., the set of formulas that are contained in $\Gamma$ at least once.
\end{restatable} 

The above lemma provides a polynomial time reduction from the provability problem of $\oc$-prenex $\{\multimap\}$-sequents in the $\{\multimap,\oc\}$-fragment of $\mathbf{ILLW}$ to deducibility in $\mathbf{BCK}$, $\mathbf{FL}^+_{\mathbf{ei}}$, $\mathbf{FL}_{\mathbf{ei}}$, and $\mathbf{FL}_{\mathbf{ew}}$; hence by Corollary~\ref{illwistowerhard} we obtain the \textsc{Tower}-hardness of deducibility in these systems. 
In the same way as before, we can show:
\begin{Lem}
	\label{prenextran1}
	Let $\vdash \wn \Gamma,\Delta$ be a $\wn$-prenex $\{\otimes,\parr,\with,\oplus,\mathbf{1},\bot\}$-sequent.
	$\vdash \wn \Gamma,\Delta$ is provable in $\mathbf{LLW}$ if and only if $\vdash \Delta$ is provable in $\mathbf{InFL}_{\mathbf{ew}}[\sigma(\Gamma^{\bot})]$.
\end{Lem}

By Corollary~\ref{llwistowerhard} and Lemma~\ref{prenextran1}, the deducibility problem for $\mathbf{InFL}_{\mathbf{ew}}$ is also hard for \tower. 
The deducibility problems for $\mathbf{BCK}$, $\mathbf{FL}^+_{\mathbf{ei}}$, $\mathbf{FL}_{\mathbf{ei}}$, $\mathbf{FL}_{\mathbf{ew}}$, and $\mathbf{InFL}_{\mathbf{ew}}$ are all in \tower\ by Corollaries~\ref{towerdeducibility} and~\ref{flewisintower}. 
We thus conclude:

\begin{Cor}\label{main}
	Each of the following decision problems is complete for \tower:
	\begin{itemize}
	    \item the deducibility problem for $\mathbf{BCK}$,
	    \item the deducibility problem for $\mathbf{FL}^+_{\mathbf{ei}}$,
	    \item the deducibility problem for $\mathbf{FL}_{\mathbf{ei}}$,
	    \item the deducibility problem for $\mathbf{FL}_{\mathbf{ew}}$,
	    \item the deducibility problem for $\mathbf{InFL}_{\mathbf{ew}}$.
	\end{itemize}
\end{Cor}

\subsection{Tower-hardness of provability in elementary affine logic}\label{towerellw}
\begin{restatable}{Lem}{llwandellw}\label{llwandellw}
	Let $\Gamma$ be a finite multiset of classical $\{\otimes,\parr,\with,\oplus,\mathbf{1},\bot\}$-formulas and $A$ a classical $\{\otimes,\parr,\with,\oplus,\mathbf{1},\bot\}$-formula. 
	$\vdash \wn \Gamma,A$ is provable in $\mathbf{LLW}$ if and only if $\vdash \wn \Gamma,\oc A$ is provable in $\mathbf{ELLW}$.
\end{restatable}
\begin{proof}
	Suppose that $\vdash \wn \Gamma, A$ is provable in $\mathbf{LLW}$. 
	By Lemma~\ref{llwtoinflei}, $\vdash \Gamma^n,A$ is provable in $\mathbf{InFL}_{\mathbf{ew}}$ for some $n$. 
	Obviously, $\vdash \Gamma^n,A$ is provable in $\mathbf{ELLW}$ for some $n$. 
	Using the rule of (F) and structural rules, we obtain a proof of $\vdash \wn \Gamma,\oc A$ in $\mathbf{ELLW}$.
	For the other direction, assume that $\vdash \wn \Gamma,\oc A$ is provable in $\mathbf{ELLW}$. 
	It is provable in $\mathbf{LLW}$ since every sequent that is provable in $\mathbf{ELLW}$ is provable in $\mathbf{LLW}$. 
	Recall that $\vdash \wn A^{\bot},A$ is provable in $\mathbf{LLW}$ whereas it is not provable in $\mathbf{ELLW}$.
	We therefore obtain a proof of $\vdash \wn \Gamma,A$ in $\mathbf{LLW}$ by a single application of (Cut).
\end{proof}

It follows from Theorem~\ref{BVASStoLL} and Lemma~\ref{llwandellw} that $\mathcal{B},Q_{\ell} \judge_{\ell} q_r,\vec 0$ if and only if $\vdash \wn \ulcorner T \urcorner,\wn Q_{\ell},\oc q_r^{\bot}$ is provable in $\mathbf{ELLW}$. 
Here the sequent $\vdash \wn \ulcorner T \urcorner,\wn Q_{\ell},\oc q_r^{\bot}$ is built from literals using only connectives in $\{\otimes,\parr,\oc,\wn\}$. 
We therefore obtain:

\begin{Cor}
	\label{prenextower1}
	The problem of determining if a given classical $\{\otimes,\parr,\oc,\wn\}$-sequent is provable in $\mathbf{ELLW}$ is  \textsc{Tower}-hard, and hence so is the provability problem for $\mathbf{ELLW}$.
\end{Cor}

Combining this with Corollary~\ref{towermembership}, we obtain: 
\begin{Cor}
	The provability problem for $\mathbf{ELLW}$ is \textsc{Tower}-complete. 
\end{Cor}

Clearly, provability in the multiplicative-exponential fragment of $\mathbf{ELLW}$ is also complete for \textsc{Tower}.
The same holds for a very small fragment of $\mathbf{IELW}$:
\begin{restatable}{Thm}{impellw}\label{impellw}
	The provability problem for the $\{\multimap,\oc\}$-fragment of $\mathbf{IELW}$ is \textsc{Tower}-complete.
\end{restatable}
\begin{proof}
    Similar to the proof of Theorem~\ref{illwistowerhard}. 
    The problem is clearly in \tower\ due to Corollary~\ref{towermembership}. 
    Let $\vdash \Gamma$ be a classical $\{\otimes,\parr,\oc,\wn\}$-sequent and $x$ a fresh variable not occurring in this sequent.
    By Theorem~\ref{illtoll5}, $\vdash \Gamma$ is provable in $\mathbf{ELLW}$ if and only if $\Gamma^{[x]} \vdash x$ is provable in $\mathbf{IELW}$. 
    By cut elimination for $\mathbf{IELW}$, we know that $\mathbf{IELW}$ is a conservative extension of its $\{\multimap,\oc\}$-fragment.
    Thus $\Gamma^{[x]} \vdash x$ is provable in $\mathbf{IELW}$ if and only if it is provable in the $\{\multimap,\oc\}$-fragment of $\mathbf{IELW}$.
    Consequently, $\vdash \Gamma$ is provable in $\mathbf{ELLW}$ if and only if $\Gamma^{[x]} \vdash x$ is provable in the $\{\multimap,\oc\}$-fragment of $\mathbf{IELW}$.
    Hence there is a polynomial time reduction from the problem of whether a given classical $\{\otimes,\parr,\oc,\wn\}$-sequent is provable in $\mathbf{ELLW}$ to provability in the $\{\multimap,\oc\}$-fragment of $\mathbf{IELW}$.
    By Corollary~\ref{prenextower1}, provability in the $\{\multimap,\oc\}$-fragment of $\mathbf{IELW}$ is hard for \tower. 
\end{proof}

\section{Concluding remarks}\label{conclusion}
We have shown the \tower-completeness of deducibility in some contraction-free substructural logics without modal operators. 
We hope that our work sheds new light on computational aspects of fuzzy logic. 
This is because $\mathbf{FL}_{\mathbf{ew}}$ forms a theoretical basis for a wide range of fuzzy logics.
In fact, one can construct from $\mathbf{FL}_{\mathbf{ew}}$ various fuzzy logics (such as monoidal t-norm based logic, basic logic, weak nilpotent minimum logic, \L ukasiewicz logic and product logic) by adding some new axioms, cf.\ \cite{EG01,Haj98}.
Remarkably, the latter four of the above examples are shown to be coNP-complete with respect to both provability and deducibility; see \cite{Han11} for a detailed survey. 
Together with those facts, our result suggests that there is a critical difference between $\mathbf{FL}_{\mathbf{ew}}$ and its fuzzy extensions with respect to computational complexity. 

We summarize the known complexity results in Tables \ref{table1} and \ref{table2}.
We adopt the notation of combinatory logic in Table~\ref{table1}:
\begin{itemize}
\item $\mathbf{BCI}$ $=$ the implicational fragment of intuitionistic linear logic,
\item $\mathbf{BCIW}$ $=$ the extension of $\mathbf{BCI}$ by the rule of contraction, 
\item $\mathbf{BCK}$ $=$ the extension of $\mathbf{BCI}$ by the rule of weakening. 
\end{itemize}
$\mathbf{BCIW}$ is usually denoted by $\mathbf{R}_{\rightarrow}$ in the relevance logic community. 
$\mathbf{IL}_{\rightarrow}$ stands for the extension of $\mathbf{BCI}$ by weakening and contraction. 
It is nothing but the implicational fragment of intuitionistic propositional logic ($\mathbf{IL}$).
In Table~\ref{table2}, $\mathbf{FL}_{\mathbf{e}}$ (resp.\ $\mathbf{FL}_{\mathbf{ec}}$) denotes the extension of $\mathbf{BCI}$ (resp.\ $\mathbf{BCIW}$) by the connectives $\otimes$, $\with$, $\oplus$, $\mathbf{1}$, and $\bot$. 

Below we comment on some results not covered in the main sections.

\begin{table}[t]
	\caption{Complexity results of extensions of $\mathbf{BCI}$.}
    \label{table1}
    \begin{center}
	\begin{tabular}{lccc}
	\hline
	& Provability & Deducibility \\
	\hline \hline
	$\mathbf{BCI}$ & \textsc{NP}-complete \cite{Bus08}& open (\textsc{Ackermann}-hard, cf.\ Corollary~\ref{ackbci})\\
	$\mathbf{BCIW}$ & 2-\textsc{ExpTime}-complete \cite{Sch216}& decidable (2-\textsc{ExpTime}-hard, cf.\ \cite{Sch216})\\
	$\mathbf{BCK}$ &\textsc{NP}-complete (Corollary~\ref{np})& \textsc{Tower}-complete (Corollary~\ref{main})\\
	$\mathbf{IL}_{\rightarrow}$ &\textsc{Pspace}-complete \cite{Sta79} & \textsc{Pspace}-complete \cite{Sta79}\\
	\hline
\end{tabular}
\end{center}
\end{table}

\begin{table}[t]
	\caption{Complexity results of extensions of $\mathbf{FL}_{\mathbf{e}}$.}
    \label{table2}
    \begin{center}
\begin{tabular}{lccc}
	\hline
	& Provability & Deducibility \\
	\hline \hline
	$\mathbf{FL}_{\mathbf{e}}$ & \textsc{Pspace}-complete \cite{LMSS92} & undecidable \cite{LMSS92} \\
	$\mathbf{FL}_{\mathbf{ec}}$ & \textsc{Ackermann}-complete, cf.\ \cite{LS15,Urq99}& \textsc{Ackermann}-complete, cf.\ \cite{LS15,Urq99}\\
	$\mathbf{FL}_{\mathbf{ei}}$ &\textsc{Pspace}-complete, cf.\ \cite{HT11}& \textsc{Tower}-complete (Corollary~\ref{main})\\
	$\mathbf{FL}_{\mathbf{ew}}$ &\textsc{Pspace}-complete, cf.\ \cite{HT11}& \textsc{Tower}-complete (Corollary~\ref{main})\\
	$\mathbf{FL}_{\mathbf{ewc}}$ ($=\mathbf{IL}$) &\textsc{Pspace}-complete \cite{Sta79} & \textsc{Pspace}-complete \cite{Sta79}\\
	\hline
\end{tabular}
\end{center}
\end{table}

\subparagraph{Ackermannian complexity and deducibility in BCI.}
In~\cite{Sch16} Schmitz also defined a non-primitive recursive complexity class by: 
$$
\textsc{Ackermann}:=\bigcup_{f \in \mathsf{FPR}} \textsc{DTime}(Ack(f(n)))
$$
where $Ack$ is the Ackermannian function and $\mathsf{FPR}$ denotes the set of primitive recursive functions.
The class contains the class \textsc{PR} of primitive recursive problems, and is closed under primitive recursive reductions. 
As in \cite{LS15,Sch16}, we define the notion of \emph{\textsc{Ackermann}-hardness} using primitive recursive reductions.

The decidability of deducibility in $\textbf{BCI}$ is a long-standing open problem.
However, one reviewer pointed out that the \textsc{Ackermann}-hardness of deducibility in \textbf{BCI} follows from the recent breakthrough by Czerwi\'nski and Orlikowski~\cite{CO21}, and Leroux~\cite{Ler21}, who independently proved that the reachability problem for VASSs is \textsc{Ackermann}-hard.

Let us sketch the proof of the Ackermannian lower bound for deducibility in $\mathbf{BCI}$. 
Clearly, by the aforementioned result in~\cite{CO21,Ler21}, reachability in BVASSs is also hard for \textsc{Ackermann}.
As is the case of lossy reachability for \abvass s, the BVASS-reachability problem is reduced to reachability in ordinary BVASSs (cf.\ Section~\ref{tower}); see~\cite[Lemma~3.5]{LS15} for details.
Furthermore, reachability in ordinary BVASSs can be efficiently encoded to the problem of whether classical linear logic ($\mathbf{LL}$), i.e., $\mathbf{LLW}$ without the rule of (W), proves a given $\wn$-prenex $\{\otimes,\parr\}$-sequent; see~\cite[Section~4.2.2]{LS15}.  
To sum up, the following holds:

\begin{Thm}[\cite{CO21,LS15,Ler21}]\label{ackLL}
The problem of determining if a given $\wn$-prenex $\{\otimes,\parr\}$-sequent is provable in $\mathbf{LL}$ is \textsc{Ackermann}-hard.
\end{Thm}

The translation in Section~\ref{simple} also holds between $\mathbf{LL}$ and $\mathbf{ILL}$ (i.e., $\mathbf{ILLW}$ without the unrestricted left-weakening); see Appendix~\ref{detailstranslation} for a sketch of a proof.

\begin{restatable}{Thm}{illtoll}\label{illtoll8}
	Let $\vdash \Gamma$ be a classical $\mathcal{L}_C$-sequent and $x$ a fresh propositional variable not occurring in $\Gamma$. $\vdash \Gamma$ is provable in $\mathbf{LL}$ if and only if $\Gamma^{[x]} \vdash x$ is provable in $\mathbf{ILL}$.
\end{restatable}

Using this, we show the intuitionistic version of Theorem~\ref{ackLL}.\footnote{In~\cite{dGS04} de~Groote, Guillaume and Salvati provided a reduction from BVASS-reachability to provability of $\oc$-prenex $\{\multimap\}$-sequents in $\mathbf{ILL}$. Thus one can also show Corollaries~\ref{ackill} and~\ref{ackbci}, using the results in \cite{CO21,dGS04,Ler21}.}
\begin{Cor}[\cite{CO21,dGS04,Ler21}]\label{ackill}
The problem of determining if a given $\oc$-prenex $\{\multimap\}$-sequent is provable
 in $\mathbf{ILL}$ is \textsc{Ackermann}-hard.
\end{Cor}
\begin{proof}
    By Theorem~\ref{ackLL}, it suffices to show that there is a reduction from the provability of $\wn$-prenex $\{\otimes,\parr\}$-sequents in $\mathbf{LL}$ to the provability of $\oc$-prenex $\{\multimap\}$-sequents in $\mathbf{ILL}$. 
    The proof is essentially the same as that of Theorem~\ref{illwistowerhard}, but we use Theorem~\ref{illtoll8} instead of Theorem~\ref{illtoll5}. 
    Let $\vdash \wn\Gamma,\Delta$ be a $\wn$-prenex $\{\otimes,\parr\}$-sequent and $x$ a fresh propositional variable not in the sequent. 
    By Theorem~\ref{illtoll8}, $\vdash \wn\Gamma,\Delta$ is provable in $\mathbf{LL}$ if and only if $(\wn\Gamma,\Delta)^{[x]} \vdash x$ is provable in $\mathbf{ILL}$.
    Obviously, the latter sequent is a $\oc$-prenex implicational sequent.
\end{proof}

\begin{Cor}\label{ackbci}
The deducibility problem for $\mathbf{BCI}$ is \textsc{Ackermann}-hard.
\end{Cor}
\begin{proof}
    In view of Corollary~\ref{ackill}, our goal is to show that: for any $\oc$-prenex $\{\multimap\}$-sequent $\oc \Gamma,\Delta \vdash C$, $\Delta \vdash C$ is provable in $\mathbf{BCI}[\sigma(\Gamma)]$ if and only if $\oc\Gamma,\Delta \vdash C$ is provable in $\mathbf{ILL}$.
    
    The proof of the \emph{if} part goes by induction on the length of the cut-free proof of $\oc\Gamma,\Delta \vdash C$ in $\mathbf{ILL}$. 
    We only analyze the case where the rule of ($\oc$D) is applied last. 
    If $\oc\Gamma,\oc A,\Delta \vdash C$ is obtained from $\oc\Gamma,A,\Delta\vdash C$ by an application of ($\oc$D), then by the induction hypothesis, $A,\Delta\vdash C$ is provable in $\mathbf{ILL}[\sigma(\Gamma)]$; thus it is also provable in $\mathbf{ILL}[\sigma(\Gamma,A)]$. 
    Since $\vdash A$ is a non-logical axiom of $\mathbf{ILL}[\sigma(\Gamma,A)]$, we may construct a derivation of $\Delta \vdash C$ in $\mathbf{ILL}[\sigma(\Gamma,A)]$ by using (Cut). 
    The remaining cases are similar.

    The converse direction is shown by induction on the length of the derivation for $\Delta \vdash C$ in $\mathbf{BCI}[\sigma(\Gamma)]$. 
    All cases are straightforward.
\end{proof}

\subparagraph{Tower-hardness of light affine logic.}
We have also shown the \tower-completeness of provability in elementary affine logic. 
On the other hand, one can define another subsystem of affine logic, called \emph{intuitionistic light affine logic} ($\mathbf{ILAL}$)~\cite{Asp98,Ter02}, which characterizes polynomial time functions. 
We obtain it from the $\{\multimap,\oc\}$-fragment of $\mathbf{ILLW}$ by dropping the rules for ($\oc$D) and ($\oc$P), and by adding a new unary connective ``$\S$'' and the following inference rules:
\begin{center}
\begin{tabular}{c@{\hspace{4em}}c}
\AxiomC{$E \vdash A$}
\RightLabel{($\oc$)}
\UnaryInfC{$\oc E \vdash \oc A$}
\DisplayProof
&
\AxiomC{$\Gamma,\Delta\vdash A$}
\RightLabel{($\S$)}
\UnaryInfC{$\oc\Gamma,\S\Delta\vdash\S A$}
\DisplayProof
\end{tabular}
\end{center}
where $E$ is a formula or the empty multiset.
It is known that provability in $\mathbf{ILAL}$ is decidable because Terui proved that it has the finite model property; see~\cite[Corollary~7.45]{Ter02}. 
We show that there is no elementary recursive algorithm for solving the provability problem for $\mathbf{ILAL}$:

\begin{Lem}
	Let $\oc \Gamma,\Delta \vdash A$ be a $\oc$-prenex $\{\multimap\}$-sequent. 
	$\oc \Gamma,\Delta \vdash A$ is provable in the $\{\multimap,\oc\}$-fragment of $\mathbf{ILLW}$ if and only if $\oc \Gamma,\S \Delta \vdash \S A$ is provable in $\mathbf{ILAL}$.
\end{Lem} 
\begin{proof}
	Our proof is inspired by the proof of the undecidability of provability in light linear logic by Terui; see~\cite[Section 2.3.2]{Ter02}. 
	For any intuitionistic $\{\multimap,\oc,\S\}$-formula $A$, we inductively define the intuitionistic $\{\multimap,\oc\}$-formula $A^-$ by $p^-=p$, $(B \multimap C)^-=B^-\multimap C^-$, $(\oc B)^-=\oc B^-$, and $(\S B)^-=B^-$.
	It is easy to see that if $\Sigma \vdash C$ is provable in $\mathbf{ILAL}$ then $\Sigma^- \vdash C^-$ is provable in the  $\{\multimap,\oc\}$-fragment of  $\mathbf{ILLW}$, for any sequent $\Sigma \vdash C$ of $\mathbf{ILAL}$. 
	This is checked by induction on the size of proofs.
	Thus the \emph{if} direction holds. 
	The \emph{only-if} direction follows from Claim (a) used in the proof of Lemma~\ref{prenextrans2} (see Appendix~\ref{appenproofhard}), the rule of $(\S)$, and the structural rules. 	
\end{proof}

By Theorem~\ref{illwistowerhard}, the provability problem for $\mathbf{ILAL}$ is hard for \tower.
We strongly believe that this problem is in \tower.

\bibliography{lipics-v2021-sample-article}

\appendix

\section{Proof of Theorem~\ref{illtoll5}}\label{detailstranslation}
For any intuitionistic $\mathcal{L}$-formula $A$, we define the classical $\mathcal{L}_C$-formula $\underline{A}$ as follows:
\begin{align*}
\underline{p}&:=p & \underline{c}&:=c\\
\underline{B \multimap C}&:=(\underline{B})^{\bot} \parr \underline{C} & \underline{B \circ C}&:=\underline{B} \circ \underline{C}\\
\underline{\oc B}&:=\oc \underline{B}
\end{align*}
where $c \in \{\mathbf{1},\bot,\top,\mathbf{0}\}$ and $\circ \in \{\otimes,\with,\oplus\}$. 

A \emph{substitution} is a map which assigns to each propositional variable a classical $\mathcal{L}_C$-formula. 
Every substitution $\tau$ uniquely extends to a map on the set of classical $\mathcal{L}_C$-formulas by $\tau(q^{\bot})=\tau(q)^{\bot}$, $\tau(c) = c$, $\tau(B \star C) = \tau(B) \star \tau(C)$, and $\tau(\iota B) =\iota\tau(B)$, where $c \in \{\mathbf{1},\bot,\top,\mathbf{0}\}$, $\star\in \{\otimes,\parr,\with,\oplus\}$ and $\iota\in\{\oc,\wn\}$. 
By induction on the height of proofs, one can show the two lemmas below:

\begin{Lem}\label{illtoll2}
	Let $\Gamma \vdash A$ be an intuitionistic $\mathcal{L}^+$-sequent. If $\Gamma \vdash A$ is provable in $\mathbf{ILLW}$ (resp.\ $\mathbf{IELW}$), then $\vdash (\underline{\Gamma})^{\bot},\underline{A}$ is provable in $\mathbf{LLW}$ (resp.\ $\mathbf{ELLW}$).
\end{Lem}

\begin{Lem}\label{illtoll4}
	Let $\vdash \Gamma$ be a classical $\mathcal{L}_C$-sequent, $\tau$ a substitution, and $\Phi$ a finite set of classical $\mathcal{L}_C$-formulas. If $\vdash \Gamma$ is provable in $\mathbf{LLW}[\Phi]$ (resp.\ $\mathbf{ELLW}[\Phi]$), then $\vdash \tau(\Gamma)$ is provable in $\mathbf{LLW}[\tau(\Phi)]$ (resp.\ $\mathbf{ELLW}[\tau(\Phi)]$).
\end{Lem}

For each classical $\mathcal{L}_C$-formula $A$, we write $\Phi_A$ for the set containing only the two formulas $A^{\bot}$ and $A \parr \mathbf{1}$. 
By straightforward induction on formulas, we can show the following two claims:

\begin{Lem}\label{ax}
	Let $A$ be a classical $\mathcal{L}_C$-formula and $F$ an intuitionistic $\mathcal{L}^+$-formula. $A^{\jump{F}},(A^{\bot})^{\jump{F}}\vdash F$ is provable in $\mathbf{ILLW}$ and $\mathbf{IELW}$.
\end{Lem}

\begin{Lem}\label{illtoll3}
	Let $A$ be a classical $\mathcal{L}_C$-formula and $F$ an intuitionistic $\mathcal{L}^+$-formula. $\vdash \underline{A^{\jump{F}}},A$ is provable in $\mathbf{LLW}[\Phi_{\underline{F}}]$ and $\mathbf{ELLW}[\Phi_{\underline{F}}]$.
\end{Lem}

Lemma~\ref{illtoll3} is an affine version of~\cite[Lemma 2.9]{Lau18}.
Using Lemma~\ref{ax}, we show:

\begin{Lem}\label{illtoll1}
	Let $\vdash \Gamma$ be a classical $\mathcal{L}_C$-sequent and $F$ an intuitionistic  $\mathcal{L}^+$-formula. If $\vdash \Gamma$ is provable in $\mathbf{LLW}$ (resp.\ $\mathbf{ELLW}$) then $\Gamma^{\jump{F}} \vdash F$ is provable in $\mathbf{ILLW}$ (resp.\ $\mathbf{IELW}$).
\end{Lem}
\begin{proof}
	By induction on the length of the cut-free proof of $\vdash \Gamma$. 
	The case where the rule of (Init) is the last step is immediate from Lemma~\ref{ax}.
	All the remaining cases are more or less straightforward. 
\end{proof}

\translation*
\begin{proof}
	We only prove the first claim, the other one being similar.
	The proof is analogous to that of \cite[Proposition 2.16]{Lau18}.
	It suffices to prove the \emph{if} direction since the converse direction is immediate from Lemma~\ref{illtoll1}.
	Suppose that $\Gamma^{[x]} \vdash x$ is provable in $\mathbf{ILLW}$. 
	By Lemma~\ref{illtoll2}, $\vdash (\underline{\Gamma^{[x]}})^{\bot},\underline{x}$ is provable in $\mathbf{LLW}$, and is clearly provable in $\mathbf{LLW}[\Phi_{\underline{x}}]$. 
	By Lemma~\ref{illtoll3} $\vdash \underline{B^{[x]}},B$ is provable in $\mathbf{LLW}[\Phi_{\underline{x}}]$ for any formula $B$ in $\Gamma$. 
    Recall that $\vdash \underline{x}^{\bot}$ is provable in   $\mathbf{LLW}[\Phi_{\underline{x}}]$. 
	Hence $\vdash \Gamma$ is provable in $\mathbf{LLW}[\Phi_{\underline{x}}]$ by several applications of (Cut). 
	Now we define a substitution $\tau$ by $\tau(x)=\bot$ and $\tau(p)=p$ for each propositional variable $p$ other than $x$. 
	By Lemma~\ref{illtoll4}, $\vdash \tau(\Gamma)$ is provable in $\mathbf{LLW}[\tau(\Phi_{\underline{x}})]$. 
	Since $\Gamma$ has no occurrences of $x$, it holds that $\tau(\Gamma)=\Gamma$; thus  $\vdash \Gamma$ is provable in $\mathbf{LLW}[\tau(\Phi_{\underline{x}})]$. 
	Obviously every sequent that is provable in $\mathbf{LLW}[\tau(\Phi_{\underline{x}})]$ is provable in $\mathbf{LLW}$; hence $\vdash \Gamma$ is provable in $\mathbf{LLW}$. 
\end{proof}

In the argument above, one might wonder why $\Phi_A$ contains the formula $A \parr \mathbf{1}$. 
Indeed, the proof of Theorem~\ref{illtoll5} works well even if one defines $\Phi_A=\{A^{\bot}\}$. 
However, the construction here allows us to prove that the translation also holds between $\mathbf{LL}$ (i.e., $\mathbf{LLW}$ without the unrestricted weakening) and $\mathbf{ILL}$ (i.e., $\mathbf{ILLW}$ without the unrestricted left-weakening). 
By repeating the argument as before, we have the following theorem:

\illtoll*
\begin{proof}
To show this, we need the linear versions of Lemmas~\ref{illtoll2}, \ref{illtoll4}, \ref{ax}, \ref{illtoll3} and \ref{illtoll1}; the proof is almost the same as before. 
However, note that the use of $A \parr \mathbf{1}$ is essential in the proof of the linear version of Lemma~\ref{illtoll3}.
\end{proof}

\section{Proof of Theorems~\ref{iezwtoabvass} and~\ref{encodingofILZW2}}\label{appmembertower1}
We firstly introduce the notion of regular deduction tree in \abvass s. 
Although this notion is implicitly introduced in \cite[Section~3.2.2]{LS15}, we fully describe it here.

Let $\mathcal{A}=\langle Q,d,T_u,T_s,T_f,T_z \rangle$ be an \abvass\ and $Q_{\ell}$ a subset of $Q$.
We inductively define \emph{$Q_{\ell} \times \{\vec 0\}$-leaf-covering regular lossy deduction trees} in $\mathcal{A}$ as follows:

\begin{itemize}
	\item For every $q \in Q_{\ell}$, $q,\vec 0$ is a $Q_{\ell} \times \{\vec 0\}$-leaf-covering regular lossy deduction tree with root $q,\vec 0$.
	\item For every $q \in Q_{\ell}$ and every $i \in \{1,\ldots,d\}$, the figure of the form:
	\begin{prooftree}\rootAtTop
		\AxiomC{}
		\noLine
		\UnaryInfC{$q,\vec 0$}
		\RightLabel{$(\mathsf{loss})$}
		\UnaryInfC{$q,\vec e_i$}
	\end{prooftree}
	is a $Q_{\ell} \times \{\vec 0\}$-leaf-covering regular lossy deduction tree with root $q,\vec e_i$.
	\item If $\mathcal{T}_1$ is a $Q_{\ell} \times \{\vec 0\}$-leaf-covering regular lossy deduction tree with root $q_1,\vec v_1$, $\mathcal{T}_2$ is a $Q_{\ell} \times \{\vec 0\}$-leaf-covering regular lossy deduction tree with root $q_2,\vec v_2$, and $q \rightarrow q_1 + q_2 \in T_s$, then the figure of the form:
	\begin{prooftree}\rootAtTop
		\AxiomC{$\mathcal{T}_1$}
		\noLine
		\UnaryInfC{$q_1,\vec v_1$}
		\AxiomC{$\mathcal{T}_2$}
		\noLine
		\UnaryInfC{$q_2,\vec v_2$}
		\RightLabel{$(q \rightarrow q_1 + q_2)$}
		\BinaryInfC{$q,\vec v_1+\vec v_2$}
	\end{prooftree}
	is a $Q_{\ell} \times \{\vec 0\}$-leaf-covering regular lossy deduction tree with root $q,\vec v_1+\vec v_2$.
	\item If $\mathcal{T}_1$ is a $Q_{\ell} \times \{\vec 0\}$-leaf-covering regular lossy deduction tree with root $q_1,\vec v$, $\mathcal{T}_2$ is a $Q_{\ell} \times \{\vec 0\}$-leaf-covering regular lossy deduction tree with root $q_2,\vec v$, and $q \rightarrow q_1 \land q_2 \in T_f$, then the figure of the form:
	\begin{prooftree}\rootAtTop
		\AxiomC{$\mathcal{T}_1$}
		\noLine
		\UnaryInfC{$q_1,\vec v$}
		\AxiomC{$\mathcal{T}_2$}
		\noLine
		\UnaryInfC{$q_2,\vec v$}
		\RightLabel{$(q \rightarrow q_1 \wedge q_2 )$}
		\BinaryInfC{$q,\vec v$}
	\end{prooftree}
	is a $Q_{\ell} \times \{\vec 0\}$-leaf-covering regular lossy deduction tree with root $q,\vec v$.
	\item If $\mathcal{T}_1$ is a $Q_{\ell} \times \{\vec 0\}$-leaf-covering regular lossy deduction tree with root $q_1,\vec  v+\vec u$ and $q \xrightarrow{\vec u} q_1\in T_u$, then the figure of the form:
	\begin{prooftree}\rootAtTop
		\AxiomC{$\mathcal{T}_1$}
		\noLine
		\UnaryInfC{$q_1,\vec v + \vec u$}
		\RightLabel{$(q \xrightarrow{\vec u} q_1)$}
		\UnaryInfC{$q,\vec v$}
	\end{prooftree}
	is a $Q_{\ell} \times \{\vec 0\}$-leaf-covering regular lossy deduction tree with root $q,\vec v$. 
	\item If $\mathcal{T}_1$ is a $Q_{\ell} \times \{\vec 0\}$-leaf-covering regular lossy deduction tree with root $q_1,\bar{\mathsf{0}}$ and $q \fz q_1\in T_z$, then the figure of the form:
	\begin{prooftree}\rootAtTop
		\AxiomC{$\mathcal{T}_1$}
		\noLine
		\UnaryInfC{$q_1,\bar{\mathsf{0}}$}
		\RightLabel{$(q \fz q_1)$}
		\UnaryInfC{$q,\bar{\mathsf{0}}$}
	\end{prooftree}
	is a  $Q_{\ell} \times \{\vec 0\}$-leaf-covering regular lossy deduction tree with root $q,\bar{\mathsf{0}}$. 
	\item If $\mathcal{T}_1$ is a $Q_{\ell} \times \{\vec 0\}$-leaf-covering regular lossy deduction tree with root $q,\vec v$ whose last applied rule is either a loss rule or a full zero test rule, then for every $i \in \{1,\ldots,d\}$, the figure of the form:
	\begin{prooftree}\rootAtTop
		\AxiomC{$\mathcal{T}_1$}
		\noLine
		\UnaryInfC{$q,\vec v$}
		\RightLabel{$(\mathsf{loss})$}
		\UnaryInfC{$q,\vec v +\vec e_i$}
	\end{prooftree}
	is a $Q_{\ell} \times \{\vec 0\}$-leaf-covering regular lossy deduction tree with root $q,\vec v+\vec e_i$.
\end{itemize}

\begin{Lem}[cf.\ Lazi\'c and Schmitz~\cite{LS15}, Section~3.2.2]\label{regular2}
	Let $\mathcal{A}=\langle Q,d,T_u,T_s,T_f,T_z \rangle$ be an \abvass\ and $\mathcal{T}$ a $Q_{\ell} \times \{\vec 0\}$-leaf-covering lossy deduction tree in $\mathcal{A}$ of the form:
	\begin{prooftree}\rootAtTop
	   \AxiomC{\raisebox{0ex}[2.0ex][0ex]{$\vdots$}}
	\noLine
	\UnaryInfC{$q,\vec v$}
	\RightLabel{$(\mathsf{loss})$}
	\UnaryInfC{$q,\vec v+\vec e_i$}
	\end{prooftree}
	where the immediate subtree with root $q,\vec v$ is regular. 
	Then there is a $Q_{\ell} \times \{\vec 0\}$-leaf-covering regular lossy deduction tree with root $q,\vec v+\vec e_i$ in $\mathcal{A}$.
\end{Lem}
\begin{proof}
	By induction on the height of $\mathcal{T}$. 
	We only analyze the case where the immediate subtree ends with an application of a split rule:
	\begin{prooftree}\rootAtTop
		\AxiomC{$\mathcal{T}_1$}
		\noLine
		\UnaryInfC{$q_1,\vec v_1$}
		\AxiomC{$\mathcal{T}_2$}
		\noLine
		\UnaryInfC{$q_2,\vec v_2$}
		\RightLabel{$(q \rightarrow q_1 + q_2)$}
		\BinaryInfC{$q,\vec v_1 + \vec v_2$}
		\RightLabel{$(\mathsf{loss})$}
		\UnaryInfC{$q,\vec v_1+\vec v_2+\vec e_i$}
	\end{prooftree}
	This can be transformed into:
	\begin{prooftree}\rootAtTop
		\AxiomC{$\mathcal{T}_1$}
		\noLine
		\UnaryInfC{$q_1,\vec v_1$}
		\AxiomC{$\mathcal{T}_2$}
		\noLine
		\UnaryInfC{$q_2,\vec v_2$}
		\RightLabel{$(\mathsf{loss})$}
		\UnaryInfC{$q_2,\vec v_2+\vec e_i$}
		\RightLabel{$(q \rightarrow q_1 + q_2)$}
		\BinaryInfC{$q,\vec v_1+\vec v_2+\vec e_i$}
	\end{prooftree}
	The height of the subtree with root $q_2,\vec v_2+ \vec e_i$ is less than that of $\mathcal{T}$. 
	Note that $\mathcal{T}_1$ and $\mathcal{T}_2$ are both regular. 
	By the induction hypothesis there is a $Q_{\ell} \times \{\vec 0\}$-leaf-covering regular lossy deduction tree with root $q_2,\vec v_2+\vec e_i$. 
	Thus we obtain the desired deduction tree. 
	The remaining cases are similar.
\end{proof}

Using the above lemma, we show:
\begin{Lem}[cf.\ Lazi\'c and Schmitz~\cite{LS15}, Section~3.2.2]\label{regular}
	Let $\mathcal{A}=\langle Q,d,T_u,T_s,T_f,T_z \rangle$ be an \abvass, $Q_{\ell}$ a subset of $Q$, and $(q,\vec v)$ in $Q \times \mathbb{N}^d$. If $\mathcal{A},Q_{\ell} \judge_{\ell} q,\vec v$, then there is a $Q_{\ell} \times \{\vec 0\}$-leaf-covering regular lossy deduction tree with root $q,\vec v$ in $\mathcal{A}$.
\end{Lem}  
\begin{proof}
	By induction on the height of the $Q_{\ell} \times \{\vec 0\}$-leaf-covering lossy deduction tree with root $q,\vec v$. 
	We carry out a case analysis on the type of last applied deduction rule.
	Clearly, the cases where rules other than $(\mathsf{loss})$ are applied last are straightforward. 
	If the root label $q,\vec v+\vec e_i$ is obtained from $q,\vec v$ by an application of $(\mathsf{loss})$, then by the induction hypothesis, there is a $Q_{\ell} \times \{\vec 0\}$-leaf-covering regular lossy deduction tree $\mathcal{T}$ with root $q,\vec v$ in $\mathcal{A}$.
	Then we may construct: 
	\begin{prooftree}\rootAtTop
		\AxiomC{$\mathcal{T}$}
		\noLine
		\UnaryInfC{$q,\vec v$}
		\RightLabel{$(\mathsf{loss})$}
		\UnaryInfC{$q,\vec v+\vec e_i$}
	\end{prooftree}
	Applying Lemma~\ref{regular2}, we obtain the desired deduction tree. 
\end{proof}

Here we recall from Section~\ref{towerupperellw} the \abvass\ $\mathcal{A}^E_F$ constructed from a formula $F$.
For each multiset $m$ over $S$, we write $\zeta(m)$ for the submultiset of formulas in $m$ of the form $\oc B$, and $\xi(m)$ for the submultiset of remaining formulas in $m$; obviously, $m=\zeta(m),\xi(m)$ and $\vec v_m=\vec v_{\zeta(m)}+\vec v_{\xi(m)}$.

We introduce the notion of standard deduction tree which plays a key role in proofs of Theorems~\ref{iezwtoabvass} and~\ref{encodingofILZW2}.
A lossy deduction tree $\mathcal{T}$ in $\mathcal{A}^E_F$ is said to be \emph{standard} if $\mathcal{T}$ satisfies the following two conditions:
\begin{enumerate}
    \item For any subtree $\mathcal{T}'$ of $\mathcal{T}$, the last applied rule in $\mathcal{T}'$ is ($\mathsf{store}$) whenever the root configuration of $\mathcal{T}'$ is of the form $q,X,\vec v_{\Gamma,\oc B}$ for some $q \in \mathcal{P}(S_{\oc})$, $\Gamma \in \mathbb{N}^{S}$, $X \in S \cup \{\bullet\}$ and $\oc B \in S_{\oc}$.
    \item For any subtree $\mathcal{T}'$ of $\mathcal{T}$, the root configuration of $\mathcal{T}'$ is in $\mathcal{P}(S_{\oc}) \times (S \cup \{\bullet\}) \times \mathbb{N}^{d}$ whenever the last applied rule in $\mathcal{T}'$ is ($\mathsf{loss}$).
\end{enumerate}
That is, a standard lossy deduction tree is a lossy deduction tree where (1) store rules are applied as close to the root configuration as possible, and (2) no applications of (\textsf{loss}) occur at the intermediate states and the leaf state.

\begin{Lem}\label{standard}
	For any $q$ in $\mathcal{P}(S_{\oc})$, $\Gamma$ in $\mathbb{N}^S$, and $X$ in $S \cup \{\bullet\}$, if $\mathcal{A}^E_F,\{q_{\ell}\} \judge_{\ell} q,X,\vec v_{\Gamma}$, then there is a $\{(q_{\ell},\vec 0)\}$-leaf-covering standard lossy deduction tree with root $q,X,\vec v_{\Gamma}$.
\end{Lem} 
\begin{proof}
	Thanks to Lemma~\ref{regular}, we can prove this lemma by induction on the height of the  $\{(q_{\ell},\vec 0)\}$-leaf-covering regular lossy deduction tree with root $q,X,\vec v_{\Gamma}$ in $\mathcal{A}^E_F$. 
	We perform a case-by-case analysis, depending on which rule is applied last. 
	
	In the case that ($\mathsf{Init1}$) is the last applied rule, the regular lossy deduction tree is of the form:
		\begin{prooftree}\rootAtTop
		\AxiomC{\raisebox{0ex}[2.0ex][0ex]{$\vdots$}}
		\noLine
		\UnaryInfC{$q_{\ell},\vec v_{\Gamma}$}
		\RightLabel{($\mathsf{Init1}$)}
		\UnaryInfC{$\emptyset,D,\vec v_{\Gamma,D}$}
	\end{prooftree}
    In this case, we can construct the following deduction tree without appealing to the induction hypothesis, as desired:
	\begin{prooftree}
		\AxiomC{$\emptyset,D,\vec v_{\Gamma,D}$}
		\doubleLine
		\RightLabel{($\mathsf{store}$)}
		\UnaryInfC{$\sigma(\zeta(\Gamma)),D,\vec v_{\xi(\Gamma),D}$}
		\doubleLine
		\RightLabel{($\mathsf{\oc W}$)}
		\UnaryInfC{$\emptyset,D,\vec v_{\xi(\Gamma),D}$}
		\doubleLine
		\RightLabel{($\mathsf{loss}$)}
		\UnaryInfC{$\emptyset,D,\vec e_{D}$}
	    \RightLabel{($\mathsf{Init1}$)}
		\UnaryInfC{$q_{\ell},\vec 0$}
	\end{prooftree}

In several cases, we need to carry out a careful analysis.
For instance, we consider here the case where a rule $q,X  \xrightarrow{-\vec e_{A \oplus B}} \textsc{Plus}(q,X,A \oplus B)$ is applied last, where $\textsc{Plus}(q,X,A \oplus B)$ denotes the intermediate state in $(\oplus\mathsf{L})$, which corresponds to the state $q,X$ and the formula $A \oplus B$.
The last part of the regular lossy deduction tree is thus of the form:
	\begin{prooftree}\rootAtTop
	\AxiomC{\raisebox{0ex}[2.0ex][0ex]{$\vdots$}}
	\noLine
	\UnaryInfC{$\textsc{Plus}(q,X,A \oplus B),\vec v_{\Gamma}$}
	\UnaryInfC{$q,X,\vec v_{\Gamma,A \oplus B}$}
\end{prooftree}
It is impossible that the configuration $\textsc{Plus}(q,X,A \oplus B),\vec v_{\Gamma}$ results from an application of a loss rule since the deduction tree is regular.
Inspecting the construction of $\mathcal{A}^E_F$, we see that the only possible outgoing transition of $\textsc{Plus}(q,X,A \oplus B)$ is the following fork rule: 
$$
\textsc{Plus}(q,X,A \oplus B) \rightarrow \textsc{Plus}_1(q,X,A \oplus B)\land \textsc{Plus}_2(q,X,A \oplus B)
$$
where $\textsc{Plus}_1(q,X,A \oplus B)$ and $\textsc{Plus}_2(q,X,A \oplus B)$ are the remaining two intermediate states in $(\oplus\mathsf{L})$.
Hence the regular lossy deduction tree must be of the form:
\begin{prooftree}\rootAtTop
	\AxiomC{\raisebox{0ex}[2.0ex][0ex]{$\vdots$}}
	\noLine
	\UnaryInfC{$\textsc{Plus}_1(q,X,A \oplus B),\vec v_{\Gamma}$}
	\AxiomC{\raisebox{0ex}[2.0ex][0ex]{$\vdots$}}
	\noLine
	\UnaryInfC{$\textsc{Plus}_2(q,X,A \oplus B),\vec v_{\Gamma}$}
	\BinaryInfC{$\textsc{Plus}(q,X,A \oplus B),\vec v_{\Gamma}$}
	\UnaryInfC{$q,X,\vec v_{\Gamma,A \oplus B}$}
\end{prooftree}
By the same argument as before, it is also impossible that $\textsc{Plus}_1(q,X,A \oplus B),\vec v_{\Gamma}$ and $\textsc{Plus}_2(q,X,A \oplus B),\vec v_{\Gamma}$ are derived by loss rules.
By the construction of $\mathcal{A}^E_F$, the only possible outgoing transition of $\textsc{Plus}_1(q,X,A \oplus B)$ is a unary rule of the form $\textsc{Plus}_1(q,X,A \oplus B) \xrightarrow{\vec e_A} q,X$.
Similarly, the only possible outgoing transition of $\textsc{Plus}_2(q,X,A \oplus B)$ is a unary rule of the form $\textsc{Plus}_2(q,X,A \oplus B) \xrightarrow{\vec e_B} q,X$.
Hence the regular lossy deduction tree is as follows:
\begin{prooftree}\rootAtTop
	\AxiomC{\raisebox{0ex}[2.0ex][0ex]{$\vdots$}}
	\noLine
	\UnaryInfC{$q,X,\vec v_{\Gamma,A}$}
	\UnaryInfC{$\textsc{Plus}_1(q,X,A \oplus B),\vec v_{\Gamma}$}
	\AxiomC{\raisebox{0ex}[2.0ex][0ex]{$\vdots$}}
	\noLine
	\UnaryInfC{$q,X,\vec v_{\Gamma,B}$}
	\UnaryInfC{$\textsc{Plus}_2(q,X,A \oplus B),\vec v_{\Gamma}$}
	\BinaryInfC{$\textsc{Plus}(q,X,A \oplus B),\vec v_{\Gamma}$}
	\UnaryInfC{$q,X,\vec v_{\Gamma,A \oplus B}$}
\end{prooftree}
We consider here the case where $A$ is in $S_{\oc}$ and $B$ is in $S\setminus S_{\oc}$, the other cases being similar.
By the induction hypothesis, there is a standard lossy deduction tree for $\mathcal{A}^E_F,\{q_{
\ell}\} \judge_{\ell} q,X,\vec v_{\Gamma,A}$. 
Such a deduction tree is of the form:
\begin{prooftree}\rootAtTop
	\AxiomC{$\mathcal{T}_1$}
	\noLine
	\UnaryInfC{$q \cup \sigma(\zeta(\Gamma))\cup\{A\},X,\vec v_{\xi(\Gamma)}$}
	\doubleLine\RightLabel{($\mathsf{store}$)}
	\UnaryInfC{$q,X,\vec v_{\Gamma,A}$}
\end{prooftree}
where the subtree $\mathcal{T}_1$ is also standard.
Again, by the induction hypothesis, there is a standard lossy deduction tree for $\mathcal{A}^E_F,\{q_{
\ell}\} \judge_{\ell} q,X,\vec v_{\Gamma,B}$. 
It must be of the form:
\begin{prooftree}\rootAtTop
	\AxiomC{$\mathcal{T}_2$}
	\noLine
	\UnaryInfC{$q \cup \sigma(\zeta(\Gamma)),X,\vec v_{\xi(\Gamma),B}$}
	\doubleLine\RightLabel{($\mathsf{store}$)}
	\UnaryInfC{$q,X,\vec v_{\Gamma,B}$}
\end{prooftree}
where the subtree $\mathcal{T}_2$ is also standard.
We can construct the standard lossy deduction tree for $\mathcal{A}^E_F,\{q_{\ell}\} \judge_{\ell} q,X,\vec v_{\Gamma,A \oplus B}$ as follows:
\begin{prooftree}\rootAtTop
	\AxiomC{$\mathcal{T}_1$}
	\noLine
	\UnaryInfC{$q\cup\sigma(\zeta(\Gamma))\cup\{A\},X,\vec v_{\xi(\Gamma)}$}
	\RightLabel{(\textsf{store})}
	\UnaryInfC{$q\cup\sigma(\zeta(\Gamma)),X,\vec v_{\xi(\Gamma),A}$}
	\UnaryInfC{$\textsc{Plus}_1(q\cup\sigma(\zeta(\Gamma)),X,A \oplus B),\vec v_{\xi(\Gamma)}$}
	\AxiomC{$\mathcal{T}_2$}
	\noLine
	\UnaryInfC{$q\cup\sigma(\zeta(\Gamma)),X,\vec v_{\xi(\Gamma),B}$}
	\UnaryInfC{$\textsc{Plus}_2(q\cup\sigma(\zeta(\Gamma)),X,A \oplus B),\vec v_{\xi(\Gamma)}$}
	\BinaryInfC{$\textsc{Plus}(q\cup\sigma(\zeta(\Gamma)),X,A \oplus B),\vec v_{\xi(\Gamma)}$}
	\UnaryInfC{$q\cup\sigma(\zeta(\Gamma)),X,\vec v_{\xi(\Gamma),A \oplus B}$}
	\doubleLine\RightLabel{(\textsf{store})}
	\UnaryInfC{$q,X,\vec v_{\Gamma,A \oplus B}$}
\end{prooftree}
The remaining cases are left to the reader.
\end{proof}

\iezwtoabvass*
\begin{proof}
The proof is similar to that of \cite[Claims 4.3.1 and 4.3.2]{LS15}.
We show the \emph{only-if} direction by induction on the length of the cut-free proof in $\mathbf{IEZW}$ of $\Theta,\Gamma \vdash \Pi$. 
We only consider the case where $\oc\Theta,\oc\Gamma \vdash \oc A$ is obtained from $\Theta,\Gamma \vdash A$ by an application of ($\oc$F), where $\Theta$ is a finite multiset of formulas from $S_{\oc}$, and $\Gamma$ is a finite multiset of formulas from $S \setminus S_{\oc}$.
By the induction hypothesis, $\mathcal{A}^E_F,\{q_{\ell}\} \judge_{\ell} \sigma(\Theta),A,\vec v_{\Gamma}$. 
By $|\sigma(\Theta)|$ applications of ($\mathsf{store}$), $\mathcal{A}^E_F,\{q_{\ell}\} \judge_{\ell} \emptyset,A,\vec v_{\Gamma}+\sum_{\oc B \in \sigma(\Theta)} \bar{\mathsf{e}}_{\oc B}$. 
By the rules for ($\mathsf{func}$), it holds that $\mathcal{A}^E_F,\{q_{\ell}\} \judge_{\ell} \sigma(\oc\Theta,\oc \Gamma),\oc A,\bar{\mathsf{0}}$.
	
In view of Lemma \ref{standard}, we show the \emph{if} direction by induction on the height of the standard lossy deduction tree for $\mathcal{A}^E_F,\{q_{\ell}\} \judge_{\ell} \sigma(\Theta),\Pi^{\dagger},\bar{\mathsf{v}}_{\Gamma}$.
We only consider the case where the configuration $\sigma(\Theta),\oc A,\vec 0$ is derived by a full zero test rule of the form $\sigma(\Theta),\oc A \fz \ts{OfCourse}(\sigma(\Theta),\oc A)$, where $\ts{OfCourse}(\sigma(\Theta),\oc A)$ denotes the intermediate state in ($\mathsf{func}$).
The following two points are crucial: 
\begin{romanenumerate}
    \item The only possible outgoing transitions of $\ts{OfCourse}(\sigma(\Theta),\oc A)$ are:
    \begin{itemize}
        \item $\ts{OfCourse}(\sigma(\Theta),\oc A) \xrightarrow{\vec e_B} \ts{OfCourse}(\sigma(\Theta),\oc A)$ (for each $\oc B \in \sigma(\Theta)$), and
        \item $\ts{OfCourse}(\sigma(\Theta),\oc A)\xrightarrow{\vec 0}\emptyset,A$.
    \end{itemize}
    \item Since the deduction tree in question is standard, it is impossible that loss rules are applied at $\textsc{OfCourse}(\sigma(\Theta),\oc A)$.
\end{romanenumerate}
The last part of the deduction tree is thus of the form: 
	\begin{prooftree}\rootAtTop
		\AxiomC{\raisebox{0ex}[2.0ex][0ex]{$\vdots$}}
		\noLine
		\UnaryInfC{$\emptyset,A,\vec v_{\Delta}$}
		\UnaryInfC{$\textsc{OfCourse}(\sigma(\Theta),\oc A),\vec v_{\Delta}$}
		\dashedLine
		\UnaryInfC{$\textsc{OfCourse}(\sigma(\Theta),\oc A),\vec 0$}
		\UnaryInfC{$\sigma(\Theta),\oc A,\vec 0$}
	\end{prooftree}
where the dashed line denotes several (possibly zero) applications of the unary rules of the form $\ts{OfCourse}(\sigma(\Theta),\oc A) \xrightarrow{\vec e_B} \ts{OfCourse}(\sigma(\Theta),\oc A)$ where $\oc B$ in $\sigma(\Theta)$.
Therefore $\Delta$ is a (possibly empty) multiset over $S$ where $\oc B$ is in $\sigma(\Theta)$ for every formula $B$ in $\sigma(\Delta)$. 
Again, since the deduction tree is standard, store rules must be successively applied as many times as possible right below the application of $\ts{OfCourse}(\sigma(\Theta),\oc A)\xrightarrow{\vec 0}\emptyset,A$.
Hence the last part of the deduction tree is actually of the form:
\begin{prooftree}\rootAtTop
	\AxiomC{\raisebox{0ex}[2.0ex][0ex]{$\vdots$}}
	\noLine
	\UnaryInfC{$\sigma(\zeta(\Delta)),A,\vec v_{\xi(\Delta)}$}
	\doubleLine
	\RightLabel{($\mathsf{store}$)}
	\UnaryInfC{$\emptyset,A,\vec v_{\Delta}$}
	\UnaryInfC{$\textsc{OfCourse}(\sigma(\Theta),\oc A),\vec v_{\Delta}$}
	\dashedLine
	\UnaryInfC{$\textsc{OfCourse}(\sigma(\Theta),\oc A),\vec 0$}
	\UnaryInfC{$\sigma(\Theta),\oc A,\vec 0$}
\end{prooftree}
By the induction hypothesis, $\zeta(\Delta),\xi(\Delta) \vdash A$, i.e., $\Delta \vdash A$, is provable in $\mathbf{IEZW}$.
By the rule of ($\oc$F), $\oc\Delta \vdash \oc A$ is provable in $\mathbf{IEZW}$. 
Due to the fact that $\sigma(\oc \Delta) \subseteq \sigma(\Theta)$, several applications of structural rules yield a proof of $\Theta \vdash \oc A$ in $\mathbf{IEZW}$, as desired.
For the other cases, see \cite[Claims~4.3.1]{LS15}.
\end{proof}

Similarly, we show:

\encodeILZWprime*
\begin{proof}
	Observe that the standardization lemma (cf.\ Lemma~\ref{standard}) holds also for $\mathcal{A}^{I'}_F$.
	The \emph{if} direction is thus shown by induction on the height of the $\{(q_{\ell},\vec 0)\}$-leaf-covering standard lossy deduction tree with root $\sigma(\Theta),\Pi^{\dagger},\vec v_{\Gamma}$ in $\mathcal{A}^{I'}_F$. 
	The proof of the \emph{only-if} direction proceeds by induction on the height of the cut-free proofs in $\mathbf{ILZW}'$. 
	We analyze here the cases concerning the right-weakening rule and the rules for $(\mathsf{W}')$; see \cite[Claims~4.3.1 and~4.3.2]{LS15} for the other cases. 
	
	If $\Theta,\Gamma \vdash A$ is obtained from $\Theta,\Gamma \vdash$ by an application of (W'), then by the induction hypothesis, $\mathcal{A}^{I'}_F,\{q_{\ell}\} \judge_{\ell} \sigma(\Theta),\bullet,\bar{\mathsf{v}}_{\Gamma}$. 
	Using the unary rule for ($\mathsf{W}'$), we obtain the desired deduction tree for $\mathcal{A}^{I'}_F,\{q_{\ell}\} \judge_{\ell} \sigma(\Theta),A,\bar{\mathsf{v}}_{\Gamma}$. 
	
	Conversely, if the standard lossy deduction tree for $\mathcal{A}^{I'}_F,\{q_{\ell}\} \judge_{\ell} \sigma(\Theta),A,\bar{\mathsf{v}}_{\Gamma}$ is obtained from the standard lossy deduction tree for $\mathcal{A}^{I'}_F,\{q_{\ell}\} \judge_{\ell} \sigma(\Theta),\bullet,\bar{\mathsf{v}}_{\Gamma}$ by an application of ($\mathsf{W}'$), then by the induction hypothesis, $\Theta,\Gamma \vdash$ is provable in $\mathbf{ILZW}'$. 
	By applying (W'), $\Theta,\Gamma \vdash A$ is provable in $\mathbf{ILZW}'$. 
\end{proof}

\section{Proof of Lemma~\ref{prenextrans2}}
\label{appenproofhard}

\prenextransii*
\begin{proof}
	For starters, observe that the following claim holds:
	\begin{description}
		\item[(a)] Let $\oc \Xi,\Sigma \vdash C$ be a $\oc$-prenex $\{\multimap\}$-sequent. If $\oc \Xi,\Sigma \vdash C$ is provable in the $\{\multimap,\oc\}$-fragment of $\mathbf{ILLW}$, then $\Xi^n,\Sigma \vdash C$ is provable in $\mathbf{BCK}$ for some $n$.
	\end{description}
    Similarly to Lemma~\ref{ilzwtoflei}, this is shown by induction on the size of cut-free proofs in the $\{\multimap,\oc\}$-fragment of $\mathbf{ILLW}$.
    To show that Statement (1) implies Statement (2), suppose that $\oc \Gamma,\Delta \vdash A$ is provable in the $\{\multimap,\oc\}$-fragment of $\mathbf{ILLW}$. 
    By Claim (a), $\Gamma^n,\Delta \vdash A$ is provable in $\mathbf{BCK}$ for some $n$.
    For each formula $B$ in $\Gamma$, $B$ is also in $\sigma(\Gamma)$. 
    Hence we obtain a proof of $\Delta \vdash A$ in $\mathbf{BCK}[\sigma(\Gamma)]$, applying (Cut) several times. 
	The implications $(2 \Rightarrow 3)$, $(3\Rightarrow 4)$ and $(4\Rightarrow 5)$ trivially hold. 
	
	We now embark on the proof of the remaining implication $(5 \Rightarrow 1)$.
	By straightforward induction on the length of proofs, we can show:
	\begin{description}
		\item[(b)] Let $\Xi$ be a finite multiset of formulas in $\mathbf{FL}_{\mathbf{ew}}$ and $\Sigma \vdash \Pi$ a sequent of $\mathbf{FL}_{\mathbf{ew}}$. If $\Sigma \vdash \Pi$ is provable in $\mathbf{FL}_{\mathbf{ew}}[\sigma(\Xi)]$, then $\oc \Xi,\Sigma \vdash \Pi$ is provable in $\mathbf{ILZW}'$.
	\end{description}
    Let us assume that $\Delta \vdash A$ is provable in $\mathbf{FL}_{\mathbf{ew}}[\sigma(\Gamma)]$. 
    By Claim (b), $\oc\Gamma,\Delta \vdash A$ is provable in $\mathbf{ILZW}'$. 
    By the cut elimination theorem for $\mathbf{ILZW}'$, we can easily check that $\mathbf{ILZW}'$ is conservative over the $\{\multimap,\oc\}$-fragment of $\mathbf{ILLW}$. 
    Hence $\oc \Gamma,\Delta \vdash A$ is provable in the $\{\multimap,\oc\}$-fragment of $\mathbf{ILLW}$.
\end{proof}
\end{document}